\documentclass[12pt,bezier]{article}
\usepackage{tikz}
\usetikzlibrary{shapes.geometric,positioning,calc,decorations.pathreplacing,fit}
\usepackage[margin=0cm]{geometry}
\usepackage{times}
\usepackage{booktabs}
\usepackage{pifont}
\usepackage{floatrow}
\usepackage{caption}
\usepackage{mathrsfs}
\usepackage{enumitem}
\usepackage[fleqn]{amsmath}
\usepackage{breqn}
\usepackage{amsfonts,amsthm,amssymb,mathrsfs,bbding}
\usepackage{txfonts}
\usepackage{graphics,multicol}
\usepackage{graphicx}
\usepackage{color}
\usepackage{amssymb}
\usepackage{amsmath}      
\usepackage{caption}
\usepackage{subcaption}
\usepackage{cite}
\usepackage{latexsym,bm}
\usepackage{indentfirst}
\usepackage{color}
\usepackage[colorlinks=true,anchorcolor=blue,filecolor=blue,linkcolor=blue,urlcolor=blue,citecolor=blue]{hyperref}
\usepackage{extarrows}
\usepackage{cite}
\usepackage{latexsym,bm}
\usepackage{mathtools}
\evensidemargin=\oddsidemargin\topmargin=-1.5cm

\newtheorem{thm}{Theorem}[section]

\newtheorem{prob}{Problem}[section]
\newtheorem{corollary}{Corollary}[section]

\newtheorem{lem}{Lemma}[section]

\theoremstyle{definition}

\addtocounter{section}{0}

\begin{document}
\title{Spectral extremal problems on $[a,b]$-factors in binding graphs, and beyond\footnote{Supported by the National Natural Science Foundation of China
{(No. 12371361)} and Distinguished Youth Foundation of Henan Province {(No. 242300421045)}.}}
\author{{\bf Ting Xu}, {\bf Ruifang Liu\thanks{Corresponding author.
E-mail addresses: xuting2001@126.com (T. Xu), rfliu@zzu.edu.cn (R. Liu).}}\\
{\footnotesize School of Mathematics and Statistics, Zhengzhou University, Zhengzhou, Henan 450001, China}}
\date{}

\maketitle
{\flushleft\large\bf Abstract}
The {\it binding number} $b(G)$ of a graph $G$ is the minimum value of $|N_{G}(X)|/|X|$ taken over all nonempty sets $X\subseteq V(G)$ such that $N_{G}(X)\neq V(G)$. A graph $G$ is {\it 1-binding} if $b(G)\geq 1$.  For integers $1\leq a \leq b$, an {\it $[a,b]$-factor} of $G$ is a spanning subgraph $H$ such that for each vertex $v \in V(G)$, $a \leq d_H(v) \leq b$. If $a=b=k$, an $[a,b]$-factor is a $k$-factor. Fan and Lin [Electron. J. Combin. 31 (2024) P1.30] characterized the unique spectral extremal graphs among 1-binding graphs with $\delta(G)\geq k$ that contain no $k$-factor for $k=1$, $2$. Subsequently, Tang and Zhang solved the problem for $k\geq 3$. Note that $\delta(G)\geq a$ is a necessary condition for the existence of an $[a,b]$-factor in a graph $G$. For $1\leq a<b$, Fan et al. [Discrete Math. 349 (2026) 115103] determined the unique spectral extremal graph $H_n^{a,b}$ among connected graphs with $\delta(G)\geq a$ that contain no $[a,b]$-factor. Note that $b(H_n^{a,b})<1$. This motivates the corresponding spectral extremal problem for 1-binding graphs. Note that every 1-binding graph with $\delta(G)\geq 1$ contains a $[1,b]$-factor for $b>1$. For $2\leq a<b$, we in this paper determine the unique extremal graph $\mathcal{B}_n^{a,b}$ with the maximum spectral radius among 1-binding graphs of order $n$ with $\delta(G)\geq a$ that contain no $[a,b]$-factor.

The {\it toughness} of a non-complete graph $G$ is $t(G)=\min\{\frac{|X|}{c(G-X)}\mid X\subseteq V(G)\, \text{and}\, c(G-X)\geq 2\},$ where $c(G-X)$ denotes the number of connected components of $G-X$. A graph $G$ is {\it 1-tough} if $c(G-X)\leq |X|$ for every $X\subseteq V(G)$ with $c(G-X)\geq 2$. Every 1-tough graph of order $n\geq2$ is 1-binding. Note that $t(\mathcal{B}_n^{a,b})<1$. This leads us to investigate the corresponding spectral extremal problem for 1-tough graphs. Note that every 1-tough graph with $\delta(G)\geq 1$ contains a $[1,b]$-factor for $b>1$. For $2\leq a<b$, we in this paper characterize the unique extremal graph with the maximum spectral radius among 1-tough graphs of order $n$ with $\delta(G)\geq a$ containing no $[a,b]$-factor.

\begin{flushleft}
\end{flushleft}

\section{Introduction}
All graphs considered in this paper are finite, simple, and undirected. Let $G$ be a graph with vertex set $V(G)$ and edge set $E(G)$. The order and size of $G$ are denoted by $|V(G)|=n$ and $|E(G)| = e(G)$. For any two vertex-disjoint graphs $G_1$ and $G_2$, let $G_1\cup G_2$ be the disjoint union of $G_1$ and $G_2$. The join $G_1\vee G_2$ is the graph obtained from $G_1\cup G_2$ by adding all possible edges between $V(G_1)$ and $V(G_2)$. For any vertex $v \in V(G)$, let $d_{G}(v)$ and $N_{G}(v)$ be the {\it degree} and neighborhood of $v$ in $G$, respectively. Let $\delta(G)=\min_{v \in V(G)} d_G(v)$. We write $u \sim v$ if $u$ is adjacent to $v$. For any subset $S$ of $V(G)$, denote by $G[S]$ the subgraph of $G$ induced by $S$. Furthermore, let $G-S$ be the induced subgraph $G[V(G)-S]$ for any subset $S\subseteq V(G)$. For two vertex disjoint subsets $S, T \subseteq V(G)$, let $|[S,T]|_G$ denote the number of edges between $S$ and $T$. For any set $S$, let $|S|$ denote the number of elements in $S$. For any $v\in V(G)$ and $S\subseteq V(G)$, we write $N_S(v)=N_G(v)\cap S$ and $d_S(v)=|N_S(v)|$. Define $N_G(X)=\cup_{v\in X}N_G(v)$ and $N_S(X)=N_G(X)\cap S$ for any $X\subseteq V(G)$.

Let $A(G)$ be the adjacency matrix of $G$, and let $\rho(G)=\lambda_{1}(G)\geq\lambda_{2}(G) \geq \cdots \geq \lambda_{n}(G)$ be its eigenvalues. By the Perron-Frobenius theorem, every connected graph $G$ has a positive eigenvector corresponding to $\rho(G)$, which is called the Perron vector of $A(G)$.

Woodall \cite{Woodall} first introduced the {\it binding number} of a graph $G$, defined by $$b(G)=\min\{\frac{|N_G(X)|}{|X|}:\emptyset\neq X\subseteq V(G), N_G(X)\neq V(G)\}.$$ A graph $G$ is called {\it $1$-binding} if $b(G)\geq 1$. An {\it $[a,b]$-factor} of $G$ is a spanning subgraph $H$ such that for each vertex $v \in V(G)$, $a \leq d_H(v) \leq b$, where $a$ and $b$ are two positive integers with $a \leq b$. If $a=b=k$, then an $[a,b]$-factor is a $k$-factor. Note that $\delta(G)\geq k$ is a necessary condition for the existence of a $k$-factor. Fan and Lin \cite{Fan3} posed the following natural problem.

\begin{prob}\label{p1.1}
Which $1$-binding graphs with $\delta(G)\geq k$ have a $k$-factor?
\end{prob}

Fan and Lin \cite{Fan3} answered Problem \ref{p1.1} for $k=1,2$. Subsequently, Tang and Zhang \cite{Tang} characterized the unique extremal graph with maximum spectral radius $G_{n,k}$ among graphs with $\delta(G)\geq k$ that contain no $k$-factor. Note that $G_{n,k}$ is 1-binding. Then the result of Tang and Zhang \cite{Tang} resolved Problem \ref{p1.1} for $k\geq 3$. Observe that a $k$-factor is a special case of an $[a,b]$-factor. Spectral conditions for the existence of $[a,b]$-factors were investigated by Fan et al. \cite{Fan5} for order $n\geq 3a+b$. Hao and Li \cite{Hao} determined the maximum spectral radius among graphs of order $n$ containing no $[a,b]$-factor, which strengthened the main result of Wei and Zhang \cite{Wei}. Observe that the binding number of the extremal graph $b(H_n^a)=\frac{n-1}{n-a+1}\geq 1$ for $a\geq 2$ in Hao and Li \cite{Hao}. Therefore $H_n^a$ is still the unique extremal graph among 1-binding graphs that contain no $[a,b]$-factor.  Note that $\delta(G)\geq a$ is a necessary condition for the existence of an $[a,b]$-factor in graphs. Hao and Li \cite{Hao} further proposed the following interesting problem.

\begin{prob}\label{p1.2}
Determine sharp lower bounds on the size or spectral radius of an n-vertex graph $G$ with $\delta(G)\geq a$ such that $G$ contains an $[a,b]$-factor.
\end{prob}

The spectral version of Problem \ref{p1.2} was solved by O \cite{O} for $a=b=1$ and by Tang and Zhang \cite{Tang} for $a=b=k\geq 2$. For $1\leq a<b$, Fan et al. \cite{Fan2} answered the spectral version of Problem \ref{p1.2} for connected graphs and obtained the following Theorem \ref{theorem1.1}.

Let $a$, $b$, and $n$ be three positive integers such that $1\leq a<b$ and $n\geq 2(a+b+2)(b+2)$. Suppose that $H_n^{a,b}$ is a graph obtained from $K_a\vee(K_{n-a-b-1}\cup(b+1)K_1)$ by adding $a-1$ edges between one vertex in $(b+1)K_1$ and $a-1$ vertices in $K_{n-a-b-1}$.

\begin{thm}[Fan et al. \cite{Fan2}]\label{theorem1.1}
Let $a$ and $b$ be two positive integers with $a<b$. If $G$ is a connected graph of order $n\geq 2(a+b+2)(b+2)$ with $\delta(G)\geq a$ which contains no $[a,b]$-factor, then $$\rho(G)\leq \rho(H_n^{a,b}),$$ with equality if and only if $G\cong H_n^{a,b}$.
\end{thm}

Note that $b(H_n^{a,b})\leq \frac{a}{b}<1$ for $1\leq a<b$. This motivates the corresponding spectral extremal problem for 1-binding graphs. For $1=a<b$, Amahashi and Kano \cite{Amahashi} proved the following theorem.

\begin{thm}[Amahashi and Kano \cite{Amahashi}]\label{theorem1.2}
Let $b\geq 2$ be an integer. A graph $G$ contains a $\{K_{1,1},K_{1,2},\dots, K_{1,b}\}$-factor if and only if $i(G-S)\leq b|S|$ for every $S\subseteq V(G)$, where $i(G-S)$ denotes the number of isolated vertices in $G-S$.
\end{thm}

Based on Theorem \ref{theorem1.2}, we prove the relation between the binding number and $[1,b]$-factors of graphs as follows.

\begin{thm}\label{theorem1.3}
Let $b\geq 2$ be an integer. If $b(G)\geq \frac{1}{b}$, then $G$ contains a $[1,b]$-factor.
\end{thm}

By Theorem \ref{theorem1.3}, every 1-binding graph $G$ with $\delta(G)\geq 1$ has a $[1,b]$-factor for $1=a<b$. This leads naturally to the following interesting problem for $2\leq a<b$.
\begin{prob}\label{p1.3}
Let $a$ and $b$ be two integers with $2\leq a<b$. What is the maximum spectral radius among $n$-vertex 1-binding graphs with $\delta(G)\geq a$ that contain no $[a,b]$-factor?
\end{prob}

Focusing on Problem \ref{p1.3}, we determine the maximum spectral radius among $n$-vertex 1-binding graphs with $\delta(G)\geq a$ that contain no $[a,b]$-factor and characterize the unique extremal graph with maximum spectral radius. To state our result, we first introduce the following graphs.

Let $a$, $b$, and $n$ be three positive integers such that $2\leq a<b$ and $n\geq (b+1)^2(2a-1)^2+2b+3$. Define vertex subsets $S_1=\{v_1, v_2, \dots, v_a\}$, $S_2=\{v_1, v_2, \dots, v_{a-1}\}$, $R_1=\{w_1, w_2, \dots, w_{n-a-b-1}\}$, $R_2=\{w_1, w_2, \dots, w_{n-a-b}\}$, $C_1=S_1\cup R_1$, $C_2=S_2\cup R_2$ and $T=\{u_1, u_2, \dots, u_{b+1}\}$. In the constructions below, an index range is understood to be empty when its lower bound exceeds its upper bound. In the figures, labels outside the ranges should be ignored, and repeated labels refer to the same vertex. Assume $2\leq a<b\leq 2a-2$. Let the graph $B_{n,1}^{a,b}$ (see Fig. \ref{fig:X}) have a vertex partition $V(B_{n,1}^{a,b})=C_1\cup T$, where $B_{n,1}^{a,b}[C_1]=K_{n-b-1}$ and $B_{n,1}^{a,b}[T]=(b+1)K_1$. Every vertex of $S_1$ is adjacent to every vertex of $T$. Define the remaining edges between $R_1$ and $T$ by
\[N_{R_1}(u_{i+1})=
\begin{cases}
\{w_1, w_2, \dots, w_{2a-b-1}\},&i=0,\\
\{w_i\},&1\le i\le\min\{b-a,2a-b-2\},\\
\{w_{i+1}\},&2a-b-1\le i\le b-a,\\
\emptyset,&b-a+1\le i\le b.
\end{cases}
\]

For $2\leq a<b$, let the graph $B_{n,2}^{a,b}$ (see Fig. \ref{fig:Y}) have a vertex partition $V(B_{n,2}^{a,b})=C_2\cup T$, where $B_{n,2}^{a,b}[C_2]=K_{n-b-1}$ and $B_{n,2}^{a,b}[T]=(b+1)K_1$. Every vertex of $S_2$ is adjacent to every vertex of $T$. If $a=2$, define the remaining edges between $R_2$ and $T$ by $N_{R_2}(u_i)=\{w_1\}$ for $i\in \{1,2\}$ and $N_{R_2}(u_{i+1})=\{w_i\}$ for $i\in\{2,3,\dots, b\}$. If $a\geq 3$, define the remaining edges between $R_2$ and $T$ by
\[N_{R_2}(u_i)=
\begin{cases}
\{w_1, w_2, \dots, w_{a-1}\},&i=1,\\
\{w_1\}, &i\in\{2,3,\dots,a+1\},\\
\{w_{i-a}\}, &a+2\leq i\leq \min\{b+1, 2a-2\},\\
\{w_{i-a+1}\}, &2a-1\leq i\leq b+1.
\end{cases}
\]
\begin{figure}[H]
    \centering
    \begin{subfigure}[b]{0.42\textwidth}
        \centering
        \includegraphics[width=\linewidth]{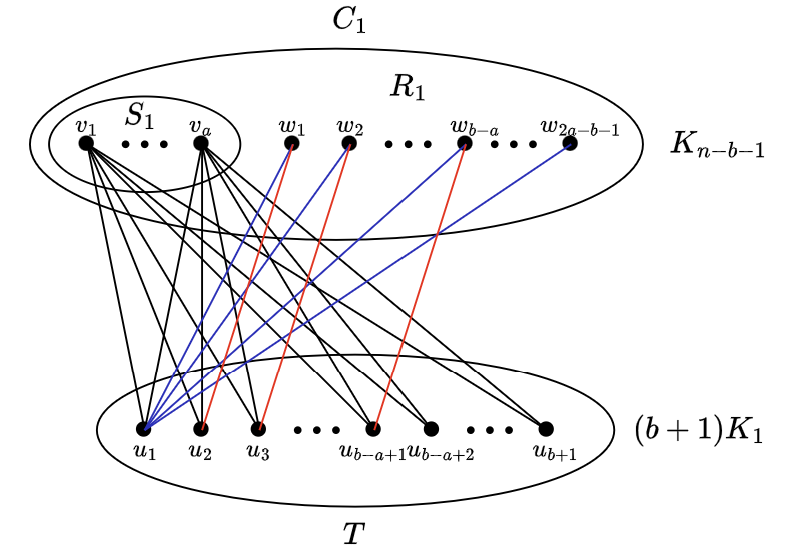}
        \caption{Graph $B_{n,1}^{a,b}$ when
        $b\leq \frac{3}{2}a-1$.}
        \label{fig:X1}
    \end{subfigure}
    \hfill
    \begin{subfigure}[b]{0.45\textwidth}
        \centering
        \includegraphics[width=\linewidth]{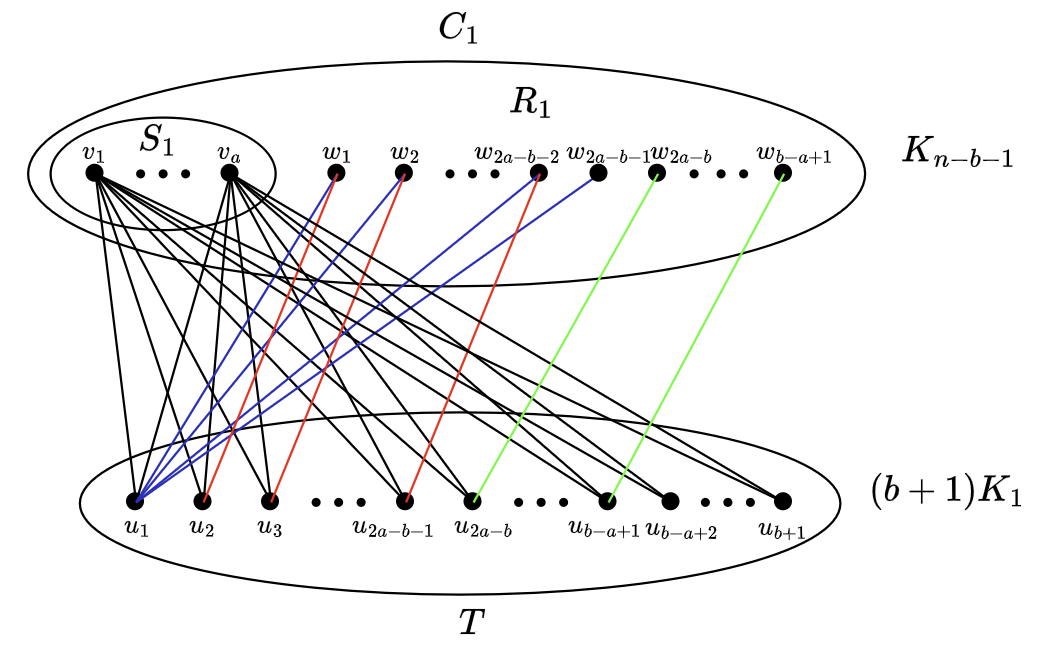}
        \caption{Graph $B_{n,1}^{a,b}$ when
        $\frac{3}{2}a-1<b\leq 2a-2$.}
        \label{fig:X2}
    \end{subfigure}
    \caption{Graph $B_{n,1}^{a,b}$.}
    \label{fig:X}
\end{figure}
\begin{figure}[htbp]
    \centering
    \begin{subfigure}[b]{0.4\textwidth}
        \centering
        \includegraphics[width=\linewidth]{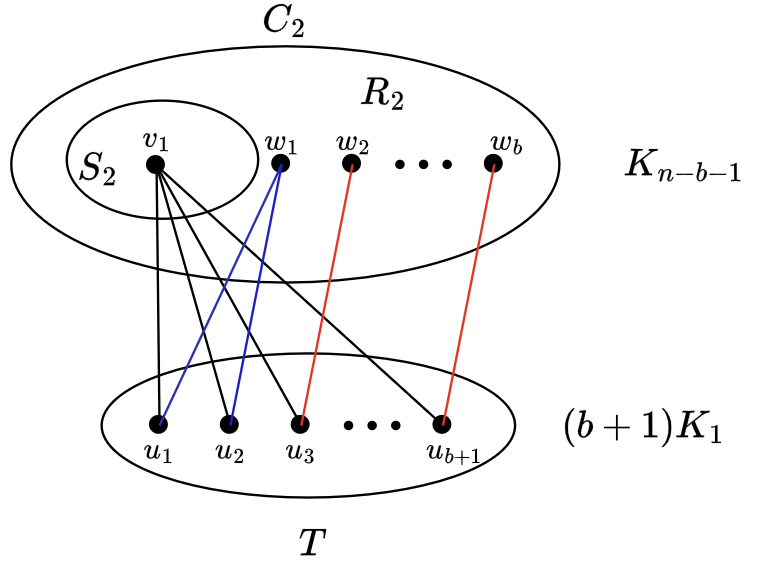}
        \caption{Graph $B_{n,2}^{a,b}$ when
        $a=2$.}
        \label{fig:Y1}
    \end{subfigure}
    \hfill
    \begin{subfigure}[b]{0.42\textwidth}
        \centering
        \includegraphics[width=\linewidth]{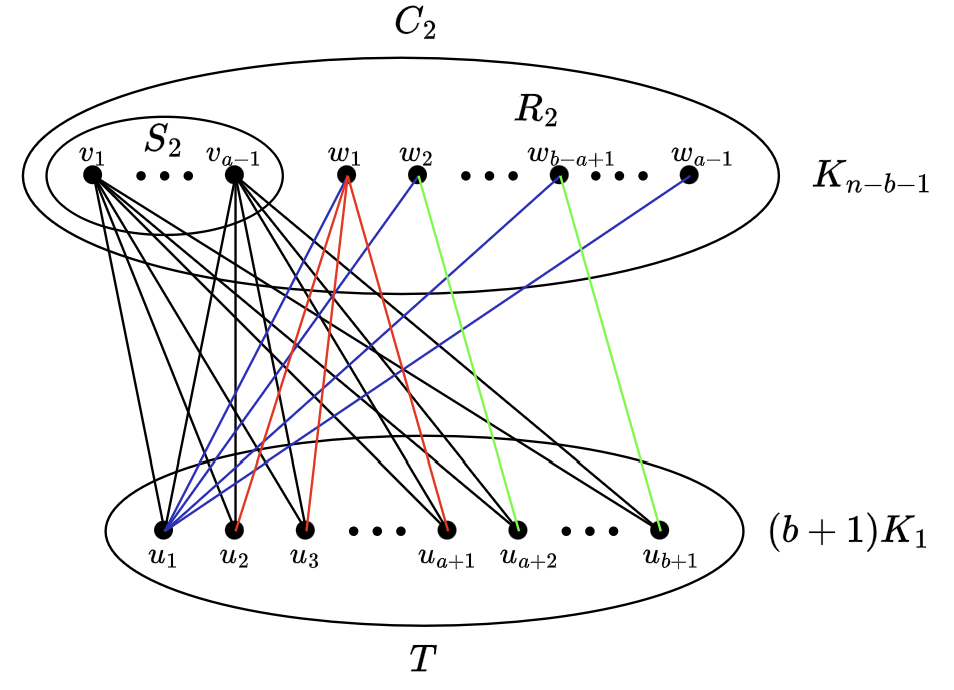}
        \caption{Graph $B_{n,2}^{a,b}$ when
        $a\geq 3$ and $b\leq 2a-3$.}
        \label{fig:Y2}
    \end{subfigure}
       \hfill
    \begin{subfigure}[b]{0.43\textwidth}
        \centering
        \includegraphics[width=\linewidth]{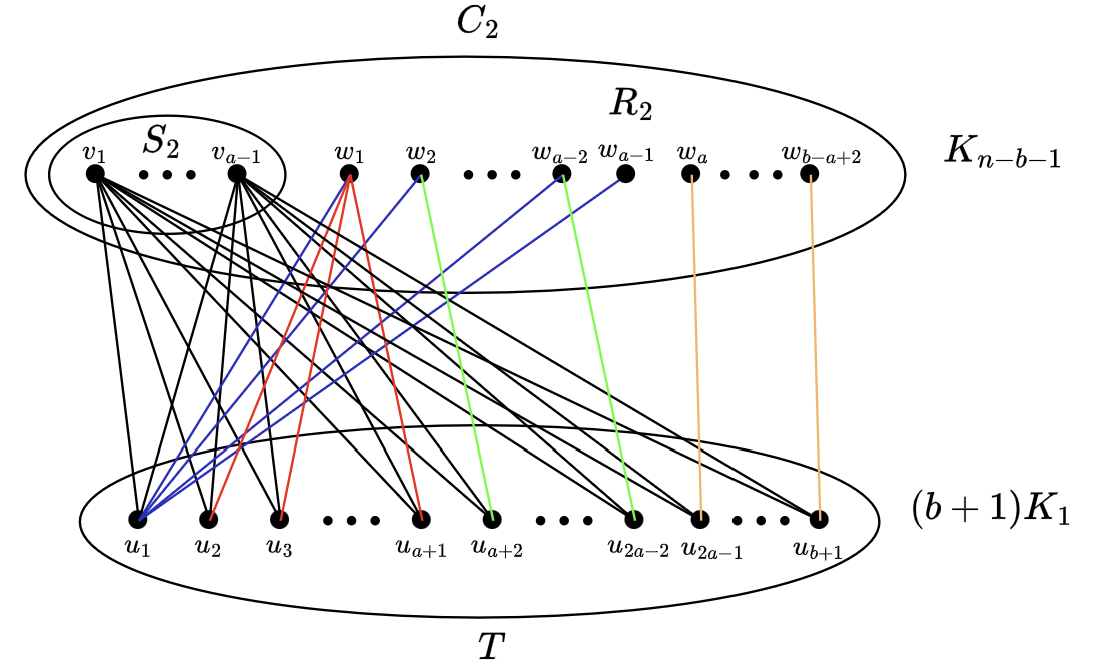}
        \caption{Graph $B_{n,2}^{a,b}$ when
        $a\geq 3$ and $b\geq 2a-2$.}
        \label{fig:Y3}
    \end{subfigure}
    \caption{Graph $B_{n,2}^{a,b}$.}
    \label{fig:Y}
\end{figure}

Among these graphs, we select $\mathcal B_n^{a,b}$ as follows and show in Theorem \ref{theorem1.4} that it is the unique spectral extremal graph for Problem \ref{p1.3}.
\[
\mathcal{B}_n^{a,b}=
\begin{cases}
B_{n,1}^{a,b}, &b\le 2a-2 \ \text{and}\ 3a^2-4ab+b^2+a+3b-3>0,\\
B_{n,2}^{a,b}, &\text{otherwise}.
\end{cases}
\]

\begin{thm}\label{theorem1.4}
Let $a$, $b$, and $n$ be three positive integers such that $2\leq a<b$ and $n\geq (b+1)^2(2a-1)^2+2b+3$. If $G$ is an $n$-vertex 1-binding graph with $\delta(G)\geq a$ which contains no $[a,b]$-factor, then $$\rho(G)\leq \rho(\mathcal{B}_n^{a,b}),$$ with equality if and only if $G\cong \mathcal{B}_n^{a,b}$.
\end{thm}

In 1973, Chv\'atal \cite{Chvatal} defined the {\it toughness} of a non-complete graph $G$ as $$t(G)=\min\{\frac{|X|}{c(G-X)}\mid X\subset V(G)\, \text{and}\, c(G-X)\geq 2\},$$ where $c(G-X)$ is the number of connected components of $G-X$. If $G\cong K_n$, $t(G)$ is defined as $n-1$. A graph $G$ is {\it $t$-tough} if $t(G)\geq t$. Goddard and Swart \cite{Goddard} proved the following theorem.
\begin{thm}[Goddard and Swart \cite{Goddard}]\label{theorem1.5}
For non-complete graphs $G$ with $b(G)\leq1$, $$t(G)\leq b(G).$$
\end{thm}

Let $G$ be an $n$-vertex graph with $t(G)\geq 1$ and $n\geq 2$. We can deduce that $b(G)\geq 1$. In fact, assume that $b(G)<1$. Since $b(K_n)=n-1\geq 1$, $G$ is non-complete. By Theorem \ref{theorem1.5}, $t(G)\leq b(G)<1,$ a contradiction. Hence we obtain the following corollary.

\begin{corollary}\label{corollary1.1}
Every 1-tough graph on at least two vertices is 1-binding.
\end{corollary}

It follows from Corollary \ref{corollary1.1} that the family of 1-tough graphs on at least two vertices is contained in the family of 1-binding graphs. Note that $t(\mathcal{B}_n^{a,b})<1$. This leads us to investigate the corresponding spectral extremal problem for 1-tough graphs. In 1990, Katerinis \cite{Katerinis} presented the following theorem.

\begin{thm}[Katerinis \cite{Katerinis}]\label{theorem1.6}
Let $a$ and $b$ be two positive integers such that $1\leq a<b$. If $t(G)\geq (a-1)+\frac{a}{b}$, then $G$ has an $[a,b]$-factor.
\end{thm}

By Theorem \ref{theorem1.6}, one can deduce the following corollary directly.

\begin{corollary}\label{corollary1.2}
Let $b>1$ be an integer. If $t(G)\geq \frac{1}{b}$, then $G$ has a $[1,b]$-factor.
\end{corollary}

It follows from Corollary \ref{corollary1.2} that every 1-tough graph with $\delta(G)\geq 1$ has a $[1,b]$-factor for $1=a<b$. Naturally, we propose the following problem for $2\leq a<b$.

\begin{prob}\label{p1.4}
Let $a$ and $b$ be two integers with $2\leq a<b$. What is the maximum spectral radius among $n$-vertex 1-tough graphs with $\delta(G)\geq a$ that contain no $[a,b]$-factor?
\end{prob}

Concerning Problem \ref{p1.4}, we determine the maximum spectral radius among $n$-vertex 1-tough graphs with $\delta(G)\geq a$ that contain no $[a,b]$-factor and characterize the unique extremal graph with maximum spectral radius. Before presenting our result, we first introduce the following graphs.

Let $a$, $b$, and $n$ be three positive integers such that $2\leq a<b$ and $n\geq (b+1)^2(2a-1)^2+2b+3$. We use the same notation $S_1$, $S_2$, $R_1$, $R_2$, $C_1$, $C_2$ and $T$ as above. The same conventions apply to the following constructions. For $2\leq a<b\leq 2a-3$, let the graph ${T}_{n,1}^{a,b}$ (see Fig. \ref{fig:X'}) have a partition $V({T}_{n,1}^{a,b})=C_1\cup T$ with ${T}_{n,1}^{a,b}[C_1]=K_{n-b-1}$ and ${T}_{n,1}^{a,b}[T]=(b+1)K_1$. Every vertex of $S_1$ is adjacent to every vertex of $T$. Define the remaining edges between $R_1$ and $T$ by
\[N_{R_1}(u_{i+1})=
\begin{cases}
\{w_1, w_2, \dots, w_{2a-b-2}\},&i=0,\\
\{w_i\},&1\le i\le\min\{b-a+1,2a-b-3\},\\
\{w_{i+1}\},&2a-b-2\le i\le b-a+1,\\
\emptyset,&b-a+2\le i\le b.
\end{cases}
\]
\begin{figure}[H]
    \centering
    \begin{subfigure}[b]{0.4\textwidth}
        \centering
        \includegraphics[width=\linewidth]{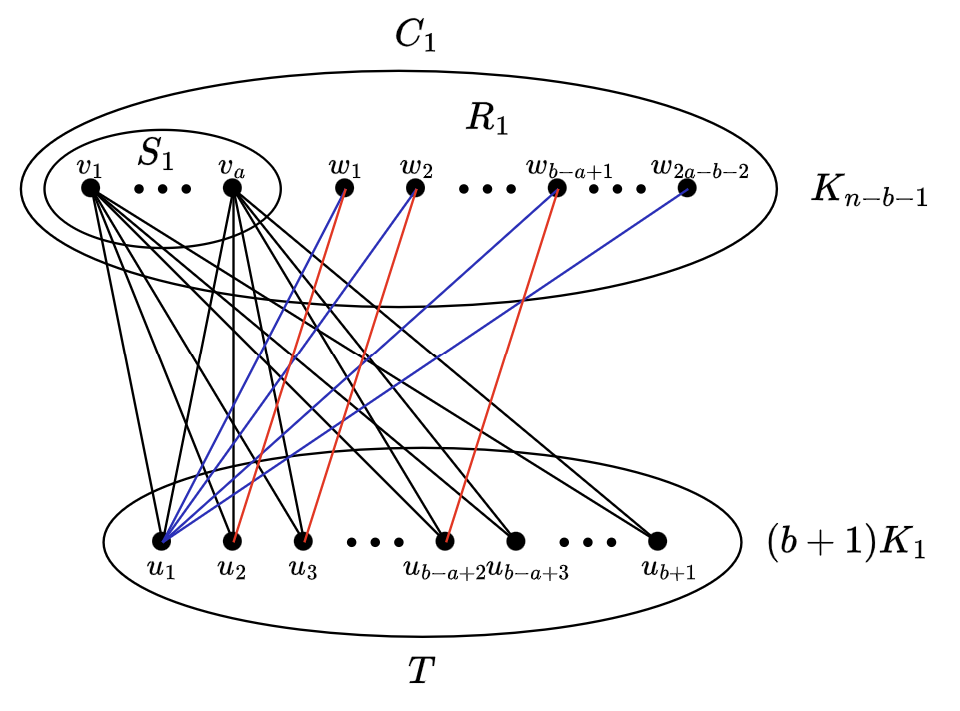}
        \caption{Graph ${T}_{n,1}^{a,b}$ when
        $b\leq \frac{3}{2}a-2$.}
        \label{fig:X1'}
    \end{subfigure}
    \hfill
    \begin{subfigure}[b]{0.45\textwidth}
        \centering
        \includegraphics[width=\linewidth]{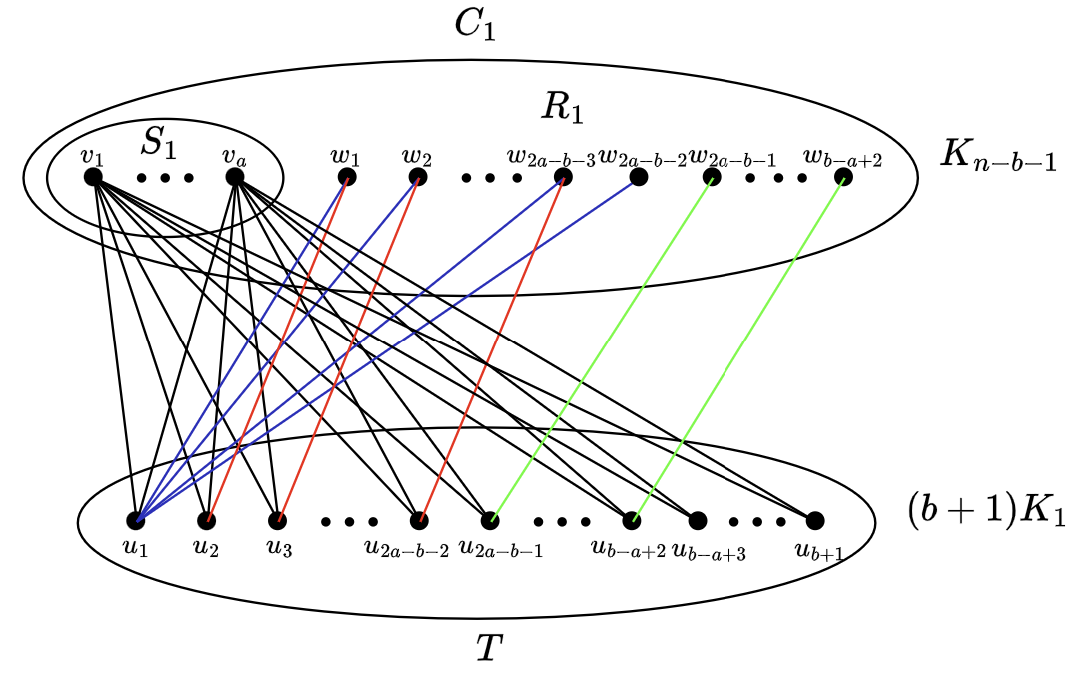}
        \caption{Graph ${T}_{n,1}^{a,b}$ when
        $\frac{3}{2}a-2<b\leq 2a-3$.}
        \label{fig:X2'}
    \end{subfigure}
    \caption{Graph ${T}_{n,1}^{a,b}$.}
    \label{fig:X'}
\end{figure}
\begin{figure}[htbp]
    \centering
    \begin{subfigure}[b]{0.42\textwidth}
        \centering
        \includegraphics[width=\linewidth]{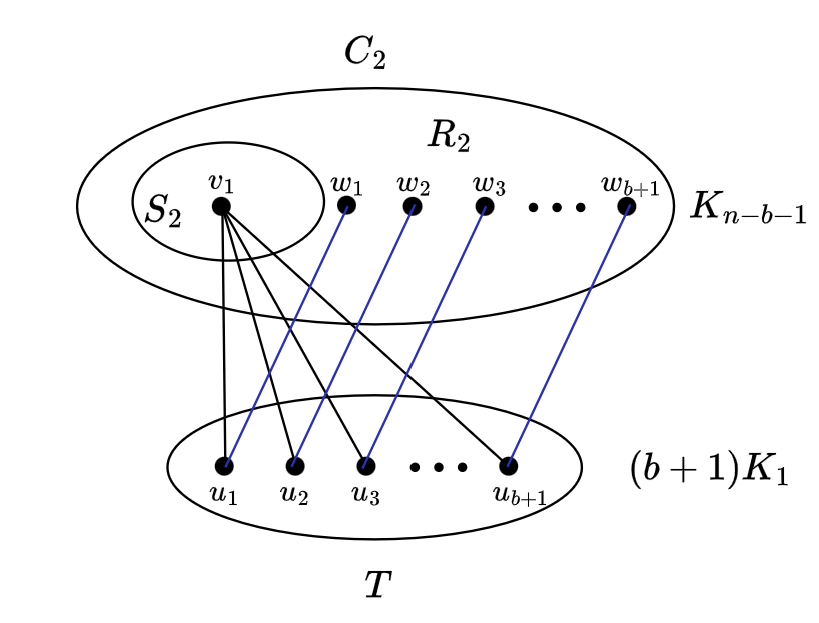}
        \caption{Graph ${T}_{n,2}^{a,b}$ when
        $a=2$.}
        \label{fig:Y1'}
    \end{subfigure}
    \hfill
    \begin{subfigure}[b]{0.43\textwidth}
        \centering
        \includegraphics[width=\linewidth]{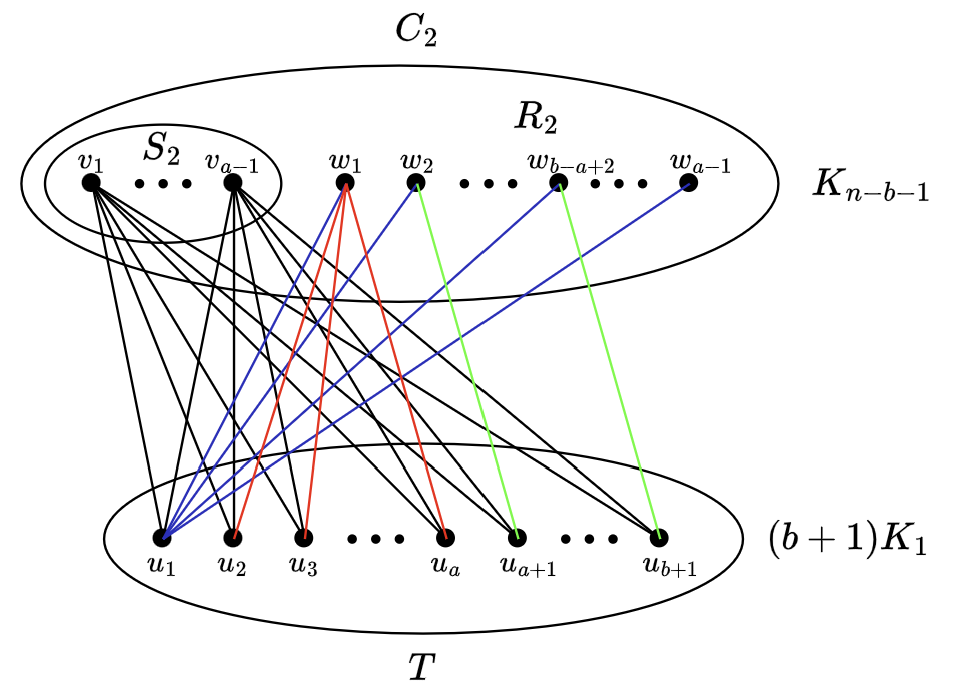}
        \caption{Graph ${T}_{n,2}^{a,b}$ when
        $a\geq 3$ and $b\leq 2a-4$.}
        \label{fig:Y2'}
    \end{subfigure}
       \hfill
    \begin{subfigure}[b]{0.44\textwidth}
        \centering
        \includegraphics[width=\linewidth]{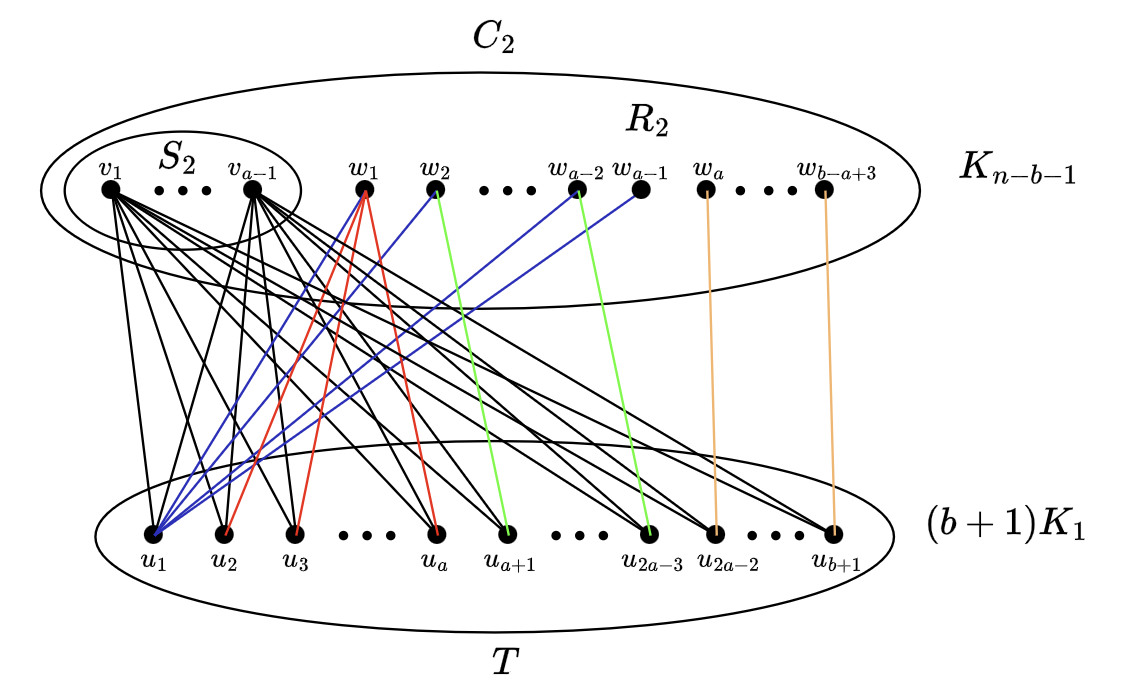}
        \caption{Graph ${T}_{n,2}^{a,b}$ when
        $a\geq 3$ and $b\geq 2a-3$.}
        \label{fig:Y3'}
    \end{subfigure}
    \caption{Graph ${T}_{n,2}^{a,b}$.}
    \label{fig:Y'}
\end{figure}

For $2\leq a<b$, let the graph ${T}_{n,2}^{a,b}$ (see Fig. \ref{fig:Y'}) have a vertex partition $V({T}_{n,2}^{a,b})=C_2\cup T$, where ${T}_{n,2}^{a,b}[C_2]=K_{n-b-1}$ and ${T}_{n,2}^{a,b}[T]=(b+1)K_1$. Every vertex of $S_2$ is adjacent to every vertex of $T$. If $a=2$, define the remaining edges between $R_2$ and $T$ by $N_{R_2}(u_i)=\{w_i\}$ for $i\in \{1,2,\dots,b+1\}$. If $a\geq 3$, define the remaining edges between $R_2$ and $T$ by
\[N_{R_2}(u_i)=
\begin{cases}
\{w_1, w_2, \dots, w_{a-1}\},&i=1,\\
\{w_1\}, &i\in\{2,3,\dots,a\},\\
\{w_{i-a+1}\}, &a+1\leq i\leq \min\{b+1, 2a-3\},\\
\{w_{i-a+2}\}, &2a-2\leq i\leq b+1.
\end{cases}
\]

We now define $\mathcal T_n^{a,b}$ by choosing between ${T}_{n,1}^{a,b}$ and ${T}_{n,2}^{a,b}$ according to the values of $a$ and $b$. Theorem \ref{theorem1.7} shows that it is the unique spectral extremal graph for Problem \ref{p1.4}.
\[
\mathcal{T}_n^{a,b}=
\begin{cases}
{T}_{n,1}^{a,b}, &b\le 2a-3 \ \text{and}\ 3a^2-4ab+b^2-3a+5b+1>0,\\
{T}_{n,2}^{a,b}, &\text{otherwise}.
\end{cases}
\]

\begin{thm}\label{theorem1.7}
Let $a$, $b$, and $n$ be three positive integers such that $2\leq a<b$ and $n\geq (b+1)^2(2a-1)^2+2b+3$. If $G$ is an $n$-vertex 1-tough graph with $\delta(G)\geq a$ which contains no $[a,b]$-factor, then $$\rho(G)\leq \rho(\mathcal{T}_n^{a,b}),$$ with equality if and only if $G\cong \mathcal{T}_n^{a,b}$.
\end{thm}

\section{Preliminaries}
In this section, we introduce some notation and auxiliary results, which are crucial to the proofs of our main results.

\begin{lem}[Brouwer and Haemers \cite{Brouwer2012}]\label{lemma2.1}
If $H$ is a subgraph of a connected graph $G$, then $\rho(H) \leq \rho(G)$ with equality if and only if $H\cong G$.
\end{lem}

In 1990, Heinrich et al. \cite{Heinrich} established the following necessary and sufficient condition for a graph to contain an $[a,b]$-factor.

\begin{lem}[Heinrich et al. \cite{Heinrich}]\label{lemma2.2}
Let $a$ and $b$ be two positive integers with $a<b.$ Then a graph $G$ contains an $[a,b]$-factor if and only if for any $S\subseteq V(G),$
\begin{eqnarray*}
a|T|-\sum_{v\in T}d_{G-S}(v)\leq b|S|,
\end{eqnarray*}
where $T=\{v\in V(G)\setminus S \mid d_{G-S}(v)<a\}.$
\end{lem}

According to Lemma \ref{lemma2.2}, one can obtain the following corollary directly.

\begin{corollary}\label{corollary2.1}
Let $a$ and $b$ be two positive integers with $a<b.$ Then a graph $G$ contains no $[a,b]$-factor if and only if there exists $S\subseteq V(G)$ such that
\begin{eqnarray}\label{equation1.1}
\sum_{v\in T}d_{G-S}(v)\leq a|T|-b|S|-1,
\end{eqnarray}
where $T=\{v\in V(G)\setminus S \mid d_{G-S}(v)<a\}.$
\end{corollary}

\begin{lem}[Wu et al. \cite{Wu2005}]\label{lemma2.3}
Let $G$ be a connected graph with $v_{i}\neq v_{j}\in V(G)$ and $\emptyset \neq S\subseteq N(v_{j})\setminus (N(v_i)\cup\{v_i\})$. Assume that $G'=G-\{v_{j}v|v\in S\}+\{v_{i}v|v\in S\}$ and $(x_{1},\ldots,x_{n})^{T}$ is the Perron vector of $A(G)$, where $x_{i}$ corresponds to $v_{i}$. If $x_{i}\geq x_{j}$, then $\rho(G')>\rho(G)$.
\end{lem}

\begin{lem}[Chv\'atal \cite{Chvatal}]\label{lemma2.4}
If $H$ is a spanning subgraph of $G$, then $t(H)\leq t(G)$.
\end{lem}

\begin{lem}\label{lemma2.5}
Let $V(G)=C\cup T$ be a vertex partition, where $G[C]$ is a clique, $T$ is an independent set, and $|C|\geq |T|+2$. Then $b(G)\geq 1$ if and only if for any $X\subseteq T$,
\begin{eqnarray}\label{equation2.1}
|N_{G}(X)|\geq |X|.
\end{eqnarray}
\end{lem}

\begin{proof}
Since $E(G[T])=\emptyset$, we have $N_{G}(X)=N_{C}(X)$ for every $X\subseteq T$. The inequality \eqref{equation2.1} is trivial when $X=\emptyset$. If $\emptyset \neq X\subseteq T$, then $N_{G}(X)\subseteq C$, and hence $N_{G}(X)\neq V(G)$. Therefore $b(G)\geq 1$ implies  $|N_{C}(X)|=|N_{G}(X)|\geq |X|$ for any $\emptyset \neq X\subseteq T$.

Conversely, assume $|N_{C}(X)|\geq |X|$ for all $X\subseteq T$. Let $\emptyset \neq Y \subseteq V(G)$ satisfy $N_{G}(Y)\neq V(G)$, and define $Y=A\cup B$, $A\subseteq C$ and $B\subseteq T$. Next we prove $|N_{G}(Y)|\geq |Y|$. If $A=\emptyset$, then we have $$|N_{G}(Y)|=|N_{C}(B)|\geq |B|=|Y|.$$ If $|A|=1$, the vertex of $A$ is adjacent to every vertex of $C\setminus A$, so $$|N_{G}(Y)|\geq |C|-1\geq |T|+1\geq |B|+1=|Y|.$$ Now assume $|A|\geq 2$ and let $M=\{x\in T: N_{C}(x)\cap A=\emptyset\}$. Since $G[C]$ is a clique and $|A|\geq 2$, the definition of $M$ gives $C\cup (T\setminus M)\subseteq N_{G}(Y)$. Together with $N_{C}(M)\subseteq C\setminus A$, we have $|M|\leq |N_{C}(M)|\leq |C\setminus A|.$ Then we can obtain that
\begin{eqnarray*}
|N_{G}(Y)|&\geq& |C|+|T|-|M|\geq |C|+|T|-|C\setminus A|\\
&=&|A|+|T|\geq |A|+|B|=|Y|.
\end{eqnarray*}
Hence $b(G)=\min\{\frac{|N_G(Y)|}{|Y|}:\emptyset\neq Y\subseteq V(G), N_G(Y)\neq V(G)\}\geq 1$.
\end{proof}

\begin{lem}\label{lemma2.6}
Let $V(G)=C\cup T$ be a vertex partition, where $G[C]$ is a clique, $T$ is an independent set, and $|C|\geq |T|+2$. Then $t(G)\geq 1$ if and only if for any $\emptyset \neq X\subseteq T$,
\begin{eqnarray}\label{equation2.2}
|N_C(X)|\geq |X|+1.
\end{eqnarray}
\end{lem}

\begin{proof}
Suppose first $t(G)\geq 1$. For any $\emptyset \neq X\subseteq T$ with $N_C(X)=C$, we have $|N_C(X)|=|C|\geq |T|+2>|X|+1.$ For any $\emptyset \neq X\subseteq T$ with $N_C(X)\neq C$, we can deduce that $|N_C(X)|\geq c(G-N_C(X))\geq |X|+1.$

Conversely, assume that for any $\emptyset \neq X\subseteq T$, \eqref{equation2.2} holds. Let $Q\subseteq V(G)$ satisfy $c(G-Q)\geq 2.$ If $C\subseteq Q$, then $c(G-Q)\leq |T|<|C|\leq |Q|$. If $C\nsubseteq Q$, let $U=\{u\in T\setminus Q: N_C(u)\subseteq Q\}$. Since $c(G-Q)\geq 2$, we have $U\neq \emptyset$. Then $c(G-Q)=|U|+1\leq |N_C(U)|\leq |Q|.$ Therefore $t(G)=\min\{\frac{|Q|}{c(G-Q)}\mid Q\subset V(G)\, \text{and}\, c(G-Q)\geq 2\}\geq 1$.
\end{proof}

\begin{lem}\label{lemma2.7}
Let $a$ and $b$ be two positive integers with $2\leq a<b.$ For $i\in \{1,2\}$, let $G_i$ be a connected graph with a partition $V(G_i)=C_i\cup T_i,$ where $|C_i|=c$, $G_i[C_i]=K_c$, $|T_i|=b+1$ and $d_{C_i}(x)\leq 2a-1$ for every $x\in T_i.$ Define $P(G_i)=\sum_{x\in T_i}d_{C_i}(x)^2$. If $c\geq (b+1)^2(2a-1)^2+b+2$, then $P(G_1)>P(G_2)$ implies $\rho(G_1)>\rho(G_2)$.
\end{lem}

\begin{proof}
Suppose that $G$ is a connected graph with a partition $V(G)=C\cup T$, where $G[C]=K_c$, $|T|=b+1$, $d_{C}(x)\leq 2a-1$ for every $x\in T$ and write $P(G)=\sum_{x\in T}d_C(x)^2$. For $r\geq c-1$ and $x,y\in T$, define $$m_r(x,y)=\boldsymbol{1}_{xy}+\frac{|N_{C}(x)\cap N_{C}(y)|}{r+1},$$
where $\boldsymbol{1}_{xy}=1$ if $xy\in E(G[T])$, otherwise $\boldsymbol{1}_{xy}=0$. For every $x\in T$, we have
\begin{eqnarray}
\sum_{y\in T}m_r(x,y)&=&d_T(x)+\frac{1}{r+1}\sum_{y\in T}|N_{C}(x)\cap N_{C}(y)|=d_T(x)+\frac{1}{r+1}\sum_{v\in N_{C}(x)}|N_T(v)|\nonumber\\
&\leq& b+\frac{(b+1)d_C(x)}{r+1}\leq b+\frac{(b+1)(2a-1)}{c}<b+1.\label{equation2.3}
\end{eqnarray}
For $k\geq 1$, by the inequality $2xy\leq x^2+y^2$ $(x,y\geq 0)$, the symmetry of $m_r$ and \eqref{equation2.3},
\begin{eqnarray}
&&\sum_{u_0, u_1, \dots,u_k\in T}d_C(u_0)d_C(u_k)\prod_{j=1}^{k}m_r(u_{j-1},u_j)\nonumber\\&\leq& \frac{1}{2}\sum_{u_0, u_1, \dots,u_k\in T}\left(d_C(u_0)^2+d_C(u_k)^2\right)\prod_{j=1}^{k}m_r(u_{j-1},u_j)\nonumber\\
&=&\frac{1}{2}\sum_{u_0\in T}d_C(u_0)^2 \left(\sum_{u_1, \dots,u_k\in T}\prod_{j=1}^{k}m_r(u_{j-1},u_j)\right)+\frac{1}{2}\sum_{u_k\in T}d_C(u_k)^2 \left(\sum_{u_0, \dots,u_{k-1}\in T}\prod_{j=1}^{k}m_r(u_{j-1},u_j)\right)\nonumber\\
&=&\sum_{u_0\in T}d_C(u_0)^2\left(\sum_{u_1, \dots, u_k\in T}\prod_{j=1}^{k}m_r(u_{j-1},u_j)\right)\nonumber\\
&=&\sum_{u_0\in T}d_C(u_0)^2\left(\sum_{u_1, \dots,u_{k-1}\in T}\prod_{j=1}^{k-1}m_r(u_{j-1},u_j)\sum_{u_k\in T}m_r(u_{k-1},u_k)\right)\nonumber\\
&\leq& (b+1)\sum_{u_0\in T}d_C(u_0)^2\left(\sum_{u_1, \dots,u_{k-1}\in T}\prod_{j=1}^{k-1}m_r(u_{j-1},u_j)\right)\leq
\cdots \leq(b+1)^kP(G).\label{equation2.4}
\end{eqnarray}
Next we define $\prod_{j=1}^0m_r(u_{j-1},u_j)=1$ and
\begin{eqnarray}\label{equation2.5}
F_G(r)=\frac{1}{r+1}\sum_{k=0}^{\infty}\frac{1}{r^{k+1}}\sum_{u_0,u_1,\dots,u_k\in T}d_C(u_0)d_C(u_k)\prod_{j=1}^km_r(u_{j-1},u_j).
\end{eqnarray}
Since $G$ is connected and $T\neq \emptyset$, we have $P(G)>0$. Clearly, $F_G(r)$ is strictly decreasing on $[c-1,\infty)$ because $m_r(x,y)$ is nonincreasing in $r$ and the term corresponding to $k=0$ in \eqref{equation2.5} is $\frac{P(G)}{r(r+1)}$, which is strictly decreasing.

Let $\rho=\rho(G)$ and $\boldsymbol{x}=(x_u)_{u\in V(G)}$ be the positive Perron vector of $A(G)$ such that $\sum_{v\in C}x_v=1$. Since $K_c$ is a subgraph of $G$, we have $\rho\geq c-1>b+1$. By $A(G)\boldsymbol{x}=\rho\boldsymbol{x}$, for $u\in C$ we have $$(\rho+1)x_u=\sum_{v\in C}x_v+\sum_{w\in N_T(u)}x_w=1+\sum_{w\in N_T(u)}x_w.$$ Summing over $u\in C$, we can obtain that \begin{eqnarray}\label{equation2.6}
\rho-c+1=\sum_{u\in C}\sum_{w\in N_T(u)}x_w=\sum_{w\in T}d_C(w)x_w.
\end{eqnarray}
For a fixed vertex $v\in T$, summing over $u\in N_C(v)$, we have $$(\rho+1)\sum_{u\in N_C(v)}x_u=d_C(v)+\sum_{u\in N_C(v)}\sum_{w\in N_T(u)}x_w=d_C(v)+\sum_{w\in T}|N_C(v)\cap N_C(w)|x_w.$$
Together with $A(G)\boldsymbol{x}=\rho\boldsymbol{x}$, this gives
\begin{eqnarray}
\rho x_v&=&\sum_{u\in N_C(v)}x_u+\sum_{u\in N_T(v)}x_u=\frac{d_C(v)}{\rho+1}+\sum_{w\in T}\frac{|N_C(v)\cap N_C(w)|x_w}{\rho+1}+\sum_{w\in T}\boldsymbol{1}_{vw}x_w\nonumber\\
&=&\frac{d_C(v)}{\rho+1}+\sum_{w\in T}m_{\rho}(v,w)x_w.\label{equation2.7}
\end{eqnarray}
Fix $u_0\in T$. Repeatedly substituting \eqref{equation2.7}, for every $N\geq 0$,
\begin{eqnarray*}
x_{u_0}&=&\frac{d_C(u_0)}{\rho(\rho+1)}+\frac{1}{\rho}\sum_{u_1\in T}m_{\rho}(u_0,u_1)x_{u_1}\\
&=&\frac{d_C(u_0)}{\rho(\rho+1)}+\frac{1}{\rho}\sum_{u_1\in T}m_{\rho}(u_0,u_1)\left[\frac{d_C(u_1)}{\rho(\rho+1)}+\frac{1}{\rho}\sum_{u_2\in T}m_{\rho}(u_1,u_2)x_{u_2}\right]\\
&=&\frac{d_C(u_0)}{\rho(\rho+1)}+\frac{1}{\rho^2(\rho+1)}\sum_{u_1\in T}d_C(u_1)m_{\rho}(u_0,u_1)+\frac{1}{\rho^2}\sum_{u_1,u_2\in T}m_{\rho}(u_0,u_1)m_{\rho}(u_1,u_2)x_{u_2}\\
&=&\cdots=\frac{1}{\rho+1}\sum_{k=0}^{N}\frac{1}{\rho^{k+1}}\sum_{u_1,\dots,u_k\in T}d_C(u_k)\prod_{j=1}^{k}m_{\rho}(u_{j-1},u_j)+R_N(u_0),
\end{eqnarray*}
where the remainder is $$R_N(u_0)=\frac{1}{\rho^{N+1}}\sum_{u_1,\dots,u_{N+1}\in T}\left(\prod_{j=1}^{N+1}m_{\rho}(u_{j-1},u_j)\right)x_{u_{N+1}}.$$
By \eqref{equation2.3} and $\rho>b+1$, we have
$$0\leq R_{N}(u_0)\leq \frac{\max_{u\in T}x_u}{\rho^{N+1}}\sum_{u_1,\dots,u_{N+1}\in T}\prod_{j=1}^{N+1}m_{\rho}(u_{j-1},u_j)\leq \max_{u\in T}x_u\left(\frac{b+1}{\rho}\right)^{N+1}\to 0, $$
and hence $\lim_{N \to \infty}R_N(u_0)=0.$ Therefore $$x_{u_0}=\frac{1}{\rho+1}\sum_{k=0}^{\infty}\frac{1}{\rho^{k+1}}\sum_{u_1,\dots,u_k\in T}d_C(u_k)\prod_{j=1}^{k}m_{\rho}(u_{j-1},u_j).$$ Multiplying by $d_C(u_0)$, summing over $u_0\in T$, and applying \eqref{equation2.5} and \eqref{equation2.6}, we obtain
\begin{eqnarray*}
\rho-c+1=F_G(\rho).
\end{eqnarray*}
Since $F_G(r)$ is strictly decreasing on $[c-1,\infty)$ and $f(r)=r-c+1$ is strictly increasing on $[c-1,\infty)$, we can deduce that $\rho$ is the unique solution of $r-c+1=F_G(r)$ in $[c-1,\infty)$.

By \eqref{equation2.4} and the nonnegativity of all summands,
\begin{eqnarray}\label{equation2.8}
\frac{P(G)}{r(r+1)}\leq F_G(r)\leq \frac{P(G)}{(r+1)(r-b-1)}.
\end{eqnarray}
Suppose $P(G_1)>P(G_2)$. Since $P(G_1)$ and $P(G_2)$ are integers, we have $P(G_1)-P(G_2)\geq 1$. For every $r\geq c-1$, it follows from \eqref{equation2.8} that \begin{eqnarray*}
F_{G_1}(r)-F_{G_2}(r)&\geq& \frac{P(G_1)}{r(r+1)}-\frac{P(G_2)}{(r+1)(r-b-1)}=\frac{r[P(G_1)-P(G_2)]-(b+1)P(G_1)}{r(r+1)(r-b-1)}\\
&\geq&\frac{r-(b+1)P(G_1)}{r(r+1)(r-b-1)}.
\end{eqnarray*}
By $P(G_1)\leq (b+1)(2a-1)^2$ and $r\geq c-1\geq (b+1)^2(2a-1)^2+b+1$, we can deduce that $F_{G_1}(r)>F_{G_2}(r)$ for every $r\geq c-1$. In particular, $$F_{G_1}(\rho(G_2))>F_{G_2}(\rho(G_2))=\rho(G_2)-c+1.$$
Since $F_{G_1}(r)$ is strictly decreasing whereas $f(r)=r-c+1$ is strictly increasing, their intersection lies to the right of $\rho(G_2)$. Hence $\rho(G_1)>\rho(G_2)$.
\end{proof}

\begin{lem}\label{lemma2.8}
Let $G$ be a finite graph, $b$ be a positive integer and $T,R\subseteq V(G)$ be disjoint vertex sets. Suppose $T=\{u_1,\dots,u_{b+1}\}$, $L=T\setminus \{u_1\}$ and $|N_R(v)|=1$ for every $v\in L$. Then
\begin{eqnarray}\label{equation2.9}
\max_{\emptyset \neq X\subseteq T}(|X|-|N_R(X)|)=\max\{b-|N_R(L)|,b+1-|N_R(T)|\}.
\end{eqnarray}
\end{lem}

\begin{proof}
For any $\emptyset\neq X\subseteq T$ with $u_1\notin X$, $|N_R(X)|\geq |N_R(L)|-(b-|X|),$ implying $$|X|-|N_R(X)|\leq b-|N_R(L)|,$$
with equality when $X=L$. For any $\emptyset\neq X\subseteq T$ with $u_1\in X$, $|N_R(X)|\geq |N_R(T)|-(b+1-|X|),$ and then
$$|X|-|N_R(X)|\leq b+1-|N_R(T)|,$$
with equality when $X=T$. Hence \eqref{equation2.9} holds.
\end{proof}

\begin{lem}\label{lemma2.9}
Let $G$ have a vertex partition $V(G)=C\cup T$, where $G[C]$ is a clique and $G[T]=(b+1)K_1$ ($b>0$). Suppose that $\emptyset\neq S\subseteq C$, every vertex of $S$ is adjacent to every vertex of $T$, and $|C|\geq |T|+2$. Let $\rho=\rho(G)$ and $\boldsymbol{x}$ be a positive Perron vector of $G$. Then for every $w\in C$,
\begin{eqnarray}\label{equation2.10}
\frac{\sum_{v\in C}x_v}{\rho+1}\leq x_w< \frac{2\sum_{v\in C}x_v}{\rho+1}.
\end{eqnarray}
If $u,v\in T$ satisfy $d_{C\setminus S}(u)\geq 2$ and $d_{C\setminus S}(v)=1$, then $x_u>x_v$. If
$u,v\in T$ satisfy $d_{C\setminus S}(u)=|S|$ and $d_{C\setminus S}(v)=1$, then $x_u<2x_v$.
\end{lem}

\begin{proof}
Since $S\neq \emptyset$ and every vertex $v\in S$ is adjacent to every vertex of $V(G)\setminus \{v\}$, the graph is connected. Note that $K_{|C|}$ is a subgraph of $G$. By Lemma \ref{lemma2.1}, $|C|\geq |T|+2$ and $|T|=b+1$, we have $\rho>\rho(K_{|C|})=|C|-1>b+1$. By $A(G)\boldsymbol{x}=\rho\boldsymbol{x}$, $\rho x_u\leq \sum_{v\in C}x_v$ for $u\in T$, and then $$\sum_{u\in T}x_u\leq (b+1)\frac{\sum_{v\in C}x_v}{\rho}<\sum_{v\in C}x_v.$$ Since $(\rho+1)x_w=\sum_{v\in C}x_v+\sum_{u\in N_T(w)}x_u$ for every $w\in C$, we have $$\frac{\sum_{v\in C}x_v}{\rho+1}\leq x_w= \frac{\sum_{v\in C}x_v}{\rho+1}+\frac{\sum_{u\in N_T(w)}x_u}{\rho+1}<\frac{2\sum_{v\in C}x_v}{\rho+1}.$$

Suppose that $u,v\in T$ satisfy $d_{C\setminus S}(u)\geq 2$ and $d_{C\setminus S}(v)=1$. By $A(G)\boldsymbol{x}=\rho\boldsymbol{x}$ and \eqref{equation2.10},
$$\rho(x_u-x_v)=\sum_{w\in N_{C\setminus S}(u)}x_w-\sum_{w\in N_{C\setminus S}(v)}x_w\geq 2\min_{w\in C}x_w-\max_{w\in C}x_w>0,$$ which implies that $x_u>x_v$.

Assume that $u,v\in T$ satisfy $d_{C\setminus S}(u)= |S|$ and $d_{C\setminus S}(v)=1$. By $A(G)\boldsymbol{x}=\rho\boldsymbol{x}$, for any $w_1\in S$ and $w_2\in C\setminus S$, we have
$(\rho+1)(x_{w_1}-x_{w_2})=\sum_{w\in T}x_w-\sum_{w\in N_T(w_2)}x_w\geq 0$. Hence $x_{w_1}\geq x_{w_2}$ for any $w_1\in S$ and $w_2\in C\setminus S$. Since $d_{C\setminus S}(u)=|S|$ and $\rho x_v>\sum_{w\in S}x_w$, it follows that
$$\rho x_u=\sum_{w\in S}x_w+\sum_{w\in N_{C\setminus S}(u)}x_w\leq 2\sum_{w\in S}x_w<2\rho x_v.$$ Therefore $x_u<2x_v$.
\end{proof}

Let $a,b,n$, and $m$ be integers satisfying $3\leq a<b$, $n\geq (b+1)^2(2a-1)^2+2b+3$, and $m\in\{b-a+1,b-a+2\}$. Consider the graph $$K_{a-1}\vee(K_{n-a-b}\cup (b+1)K_1),$$ where $S$, $R$, and $T=\{u_1,\dots,u_{b+1}\}$ are the vertex sets corresponding to $K_{a-1}$, $K_{n-a-b}$ and $(b+1)K_1$, respectively. Let $L=T\setminus \{u_1\}$. Define $\mathcal{G}_n^{a,b}(m)$ to be the family of graphs obtained from this graph by adding edges between $R$ and $T$ such that $d_R(u_1)=a-1$, $d_R(v)=1$ for every $v\in L$, $|N_R(L)|\geq m$, and $|N_R(T)|\geq m+1$.

\begin{lem}\label{lemma2.10}
Let $G^*\in \mathcal{G}_n^{a,b}(m)$ have maximum spectral radius, and let $V(G^*)=S\cup R\cup \{u_1\}\cup L$. Then $|N_R(L)|=m$, and there is a unique $w\in N_R(u_1)\cap N_R(L)$ such that $d_L(w)=b-m+1$. Every other vertex of $N_R(L)$ has exactly one neighbor in $L$, and $|N_R(u_1)\cap N_R(L)|=\min\{m,a-2\}$. If $m=b-a+1$, then $G^*\cong B_{n,2}^{a,b}$. If $m=b-a+2$, then $G^*\cong {T}_{n,2}^{a,b}$.
\end{lem}

\begin{proof}
Recall the definitions of $B_{n,2}^{a,b}$ and ${T}_{n,2}^{a,b}$. It is easy to check that $B_{n,2}^{a,b}\in \mathcal{G}_n^{a,b}(b-a+1)$ and ${T}_{n,2}^{a,b}\in \mathcal{G}_n^{a,b}(b-a+2)$ when $3\leq a<b$. Since $S\neq \emptyset$ and every vertex $v\in S$ is adjacent to every vertex of $V(G^*)\setminus \{v\}$, $G^*$ is connected. Let $\rho=\rho(G^*)$ and $\boldsymbol{x}$ be a positive Perron vector of $G^*$. Suppose that $w_p\in N_{R}(L)$ satisfies $x_{w_p}=\max_{w\in N_{R}(L)}x_w$.

Now we prove that $d_L(w)=1$ for every vertex $w\in N_{R}(L)\setminus \{w_p\}$. Suppose to the contrary that there exists $w_j\in N_{R}(L)\setminus \{w_p\}$ such that $d_L(w_j)\geq 2$. Let $N_{R}(u_i)=\{w_j\}$ where $u_i\in L$. Construct $G'=G^*-u_iw_j+u_iw_p$. Observe that $G'\in \mathcal{G}_n^{a,b}(m)$. By $x_{w_p}\geq x_{w_j}$ and Lemma \ref{lemma2.3}, $\rho<\rho(G')$, which contradicts the maximality of $\rho$. Hence for every vertex $w\in N_{R}(L)\setminus \{w_p\}$, we have
\begin{eqnarray}\label{equation2.11}
d_L(w)=1.
\end{eqnarray}

Next we claim that $|N_{R}(L)|=m$. Suppose to the contrary that $|N_{R}(L)|\geq m+1$. Since $m+1\geq 2$, we can obtain that there exists a vertex $w_j\in N_{R}(L)\setminus \{w_p\}$. By \eqref{equation2.11}, there is a unique vertex $u_i\in L$ such that $N_R(u_i)=\{w_j\}$. Let $G'=G^*-u_iw_j+u_iw_p$. If $|N_R(\{u_1\}\cup L)|\geq m+2$, then $G'\in \mathcal{G}_n^{a,b}(m)$. By $x_{w_p}\geq x_{w_j}$ and Lemma \ref{lemma2.3}, $\rho<\rho(G')$, contrary to the maximality of $\rho$. If $|N_R(\{u_1\}\cup L)|=m+1$, then $|N_{R}(L)|=m+1$ and $N_R(u_1)\subseteq N_{R}(L)$. Since $|S|=a-1\geq 2$, there exists $w_j\in N_R(u_1)\setminus \{w_p\}\subseteq N_R(L)\setminus \{w_p\}$. By \eqref{equation2.11}, there is a vertex $u_i\in L$ such that $N_L(w_j)=\{u_i\}$. Let $G'=G^*-u_iw_j+u_iw_p$. Obviously, $G'\in \mathcal{G}_n^{a,b}(m)$. Combining $x_{w_p}\geq x_{w_j}$ and Lemma \ref{lemma2.3}, $\rho<\rho(G')$, contradicting the maximality of $\rho$. Hence $|N_{R}(L)|=m$. Combining this with \eqref{equation2.11}, we have $$b=|L|=\sum_{w\in N_R(L)}d_L(w)=d_L(w_p)+m-1,$$ which implies that $d_L(w_p)=b-m+1$.

Now we assert $w_p\in N_R(u_1)$. Otherwise, suppose $w_p\notin N_R(u_1)$. If $N_R(L)\cap N_R(u_1)\neq \emptyset$, there exists $w_j\in N_R(L)\cap N_R(u_1)$. Let $G'=G^*-u_1w_j+u_1w_p$. Note that $G'\in \mathcal{G}_n^{a,b}(m)$. According to $x_{w_p}\geq x_{w_j}$ and Lemma \ref{lemma2.3}, $\rho<\rho(G')$, which contradicts the maximality of $\rho$. Assume $N_R(L)\cap N_R(u_1)=\emptyset$. Let $w_l\in N_R(u_1)$ and $N_R(u_i)=\{w_p\}$ where $u_i\in L$. By Lemma \ref{lemma2.9}, $A(G^*)\boldsymbol{x}=\rho\boldsymbol{x}$, $d_L(w_p)=b-m+1\geq 2$ and symmetry, $$(\rho+1)(x_{w_p}-x_{w_l})=(b-m+1)x_{u_i}-x_{u_1}\geq 2x_{u_i}-x_{u_1}>0,$$ implying $x_{w_p}>x_{w_l}$. Let $G'=G^*-u_1w_l+u_1w_p$. Clearly, $G'\in \mathcal{G}_n^{a,b}(m)$. By $x_{w_p}> x_{w_l}$ and Lemma \ref{lemma2.3}, $\rho<\rho(G')$, contrary to the maximality of $\rho$. Hence $w_p\in N_R(u_1)$.

Since $|N_R(L)|=m$, $|N_R(u_1)|=a-1$ and $|N_R(\{u_1\}\cup L)|\geq m+1$, we have $$|N_R(u_1)\cap N_R(L)|=|N_R(u_1)|+|N_R(L)|-|N_R(\{u_1\}\cup L)|\leq a-2.$$ Together with $|N_R(L)|=m$, this gives $|N_R(u_1)\cap N_R(L)|\leq\min\{m,a-2\}$. If $|N_R(u_1)\cap N_R(L)|<\min\{m,a-2\}$, there exist $w_j\in N_R(L)\setminus N_R(u_1) $ and $w_l\in N_R(u_1)\setminus N_R(L)$. According to \eqref{equation2.11}, there is $u_i\in L$ such that $N_{L}(w_j)=\{u_i\}$. By $A(G^*)\boldsymbol{x}=\rho\boldsymbol{x}$ and Lemma \ref{lemma2.9}, $$(\rho+1)(x_{w_l}-x_{w_j})=x_{u_1}-x_{u_i}>0,$$ and hence $x_{w_l}>x_{w_j}$. Construct $G'=G^*-u_iw_j+u_iw_l$. Note that $G'\in \mathcal{G}_n^{a,b}(m)$. By $x_{w_l}>x_{w_j}$ and Lemma \ref{lemma2.3}, $\rho<\rho(G')$, contradicting the maximality of $\rho$. Therefore $|N_R(u_1)\cap N_R(L)|=\min\{m,a-2\}$.

Combining \eqref{equation2.11}, $|N_R(L)|=m$, $w_p\in N_R(u_1)\cap N_R(L)$, and $|N_R(u_1)\cap N_R(L)|=\min\{m,a-2\}$, we can obtain that  $G^*\cong B_{n,2}^{a,b}$ when $m=b-a+1$, and $G^*\cong {T}_{n,2}^{a,b}$ when $m=b-a+2$.
\end{proof}

Let $a$, $b$, and $n$ be three positive integers such that $2\leq a<b$. Let $\mathcal{B}_n(a,b)$ be the family of $n$-vertex 1-binding graphs $G$ such that $\delta(G)\geq a$ and $G$ has no $[a,b]$-factor.

\begin{lem}\label{lemma2.11}
Let $a$, $b$, and $n$ be three positive integers such that $2\leq a<b$ and $n\geq (b+1)^2(2a-1)^2+2b+3$. Then $B_{n,2}^{a,b}\in \mathcal{B}_n(a,b)$. If $b\leq2a-2$, then $B_{n,1}^{a,b}\in \mathcal{B}_n(a,b)$.
\end{lem}

\begin{proof}
Suppose first that $b\leq 2a-2$. We prove $B_{n,1}^{a,b}\in \mathcal{B}_n(a,b)$. Recall the definition of $B_{n,1}^{a,b}$. Let $V_1=S_1$, $V_2=T$, and $V_3=R_1$. Take $S=V_1$. Observe that $d_{B_{n,1}^{a,b}-S}(x)\geq |R_1|-1>a$ for every $x\in R_1$. Since the possible degrees of a vertex of $V_2$ in $B_{n,1}^{a,b}-S$ are $0$, $1$, and $2a-b-1\leq a-2<a$, we can obtain that $V_2=\{v\in V(B_{n,1}^{a,b})\setminus S \mid d_{B_{n,1}^{a,b}-S}(v)<a\}$. Then $$a|V_2|-\sum_{v\in V_2}d_{B_{n,1}^{a,b}-S}(v)=ab+1>ab=b|S|.$$ By Lemma \ref{lemma2.2}, $B_{n,1}^{a,b}$ has no $[a,b]$-factor. Let $P=\{u_1, u_2, \dots, u_{b-a+1}\}$ and $Z=T\setminus P$. For every nonempty $W\subseteq T$, $|N_{R_1}(W\cap P)|\geq |W\cap P|$. Since $N_{R_1}(x)=\emptyset$ for $x\in Z$ and $|W\cap Z|\leq a=|S|$, we have $$|N_{C_1}(W)|=a+|N_{R_1}(W\cap P)|\geq a+|W\cap P|\geq |W\cap Z|+|W\cap P|=|W|.$$ By Lemma \ref{lemma2.5}, we obtain that $b(B_{n,1}^{a,b})\geq1$. Note that $\delta(B_{n,1}^{a,b})=a$. Hence $B_{n,1}^{a,b}\in \mathcal{B}_n(a,b)$.

Now we prove $B_{n,2}^{a,b}\in \mathcal{B}_n(a,b)$. Recall the definition of $B_{n,2}^{a,b}$. Let $V_1'=S_2$, $V_2'=T$, and $V_3'=R_2$. Take $S=V_1'$. Note that the possible degrees of a vertex of $V_2'$ in $B_{n,2}^{a,b}-S$ are $1$ and $a-1$. Since $d_{B_{n,2}^{a,b}-S}(x)\geq |R_2|-1>a$ for every $x\in R_2$, we have $V_2'=\{v\in V(B_{n,2}^{a,b})\setminus S \mid d_{B_{n,2}^{a,b}-S}(v)<a\}$. Then $$a|V_2'|-\sum_{v\in V_2'}d_{B_{n,2}^{a,b}-S}(v)=ab-b+1>ab-b=b|S|.$$ By Lemma \ref{lemma2.2}, $B_{n,2}^{a,b}$ contains no $[a,b]$-factor. If $a=2$, we can deduce that $|N_{R_2}(W)|\geq |W|-1$ for every nonempty $W\subseteq T$. According to $|S|=1$, we have $|N_{C_2}(W)|=1+|N_{R_2}(W)|\geq|W|.$ By Lemma \ref{lemma2.5}, we obtain that $b(B_{n,2}^{a,b})\geq1$. Now assume that $a\geq 3$. Let $L_1=\{u_2, u_3,\dots, u_{a+1}\}$, $L_2=\{u_{a+2},\dots, u_{\min\{b+1,2a-2\}}\}$ and $L_3=\{u_{2a-1},\dots, u_{b+1}\}$. Let $\emptyset \neq W\subseteq T$. If $|W\cap L_1|=0$, then $|N_{R_2}(W)|\geq |W|$. If $|W\cap L_1|\geq 1$, then $|N_{R_2}(W)|\geq|W|-(|W\cap L_1|-1)\geq |W|-(a-1).$ In both cases, $|N_{C_2}(W)|=(a-1)+|N_{R_2}(W)|\geq |W|.$ By Lemma \ref{lemma2.5}, we obtain that $b(B_{n,2}^{a,b})\geq1$. Note that $\delta(B_{n,2}^{a,b})=a$. So $B_{n,2}^{a,b}\in \mathcal{B}_n(a,b)$.
\end{proof}

Let $a$, $b$, and $n$ be three positive integers such that $2\leq a<b$. Let $\mathcal{T}_n(a,b)$ be the family of $n$-vertex 1-tough graphs $G$ such that $\delta(G)\geq a$ and $G$ has no $[a,b]$-factor.

\begin{lem}\label{lemma2.12}
Let $a$, $b$, and $n$ be three positive integers such that $2\leq a<b$ and $n\geq (b+1)^2(2a-1)^2+2b+3$. Then ${T}_{n,2}^{a,b}\in \mathcal{T}_n(a,b)$. If $b\leq 2a-3$, then ${T}_{n,1}^{a,b}\in \mathcal{T}_n(a,b)$.
\end{lem}

\begin{proof}
Recall the definitions of ${T}_{n,1}^{a,b}$ and ${T}_{n,2}^{a,b}$. In the following, ${T}_{n,1}^{a,b}$ is considered when $b\leq 2a-3$. Note that $V({T}_{n,1}^{a,b})=C_1\cup T$ where $C_1=S_1\cup\ R_1$, $V({T}_{n,2}^{a,b})=C_2\cup T$ where $C_2=S_2\cup\ R_2$, $a|T|-\sum_{u\in T}d_{{T}_{n,1}^{a,b}-S_1}(u)=b|S_1|+1$, and $a|T|-\sum_{u\in T}d_{{T}_{n,2}^{a,b}-S_2}(u)=b|S_2|+1$. It is easy to check that $$T=\{u\in V({T}_{n,1}^{a,b})\setminus S_1: d_{{T}_{n,1}^{a,b}-S_1}(u)<a\}=\{u\in V({T}_{n,2}^{a,b})\setminus S_2: d_{{T}_{n,2}^{a,b}-S_2}(u)<a\}.$$ By Corollary \ref{corollary2.1}, both ${T}_{n,1}^{a,b}$ and ${T}_{n,2}^{a,b}$ have no $[a,b]$-factor.

For ${T}_{n,1}^{a,b}$, let $Z=\{u\in T: N_{R_1}(u)=\emptyset\}$. Observe that $|Z|=a-1$ and $|N_{R_1}(X)|\geq |X|$ for every $X\subseteq T\setminus Z$. Then for every $\emptyset \neq X\subseteq T$, $$|N_{C_1}(X)|=a+|N_{R_1}(X\setminus Z)|\geq a+|X\setminus Z|\geq |X|+1>|X|.$$ It follows from Lemma \ref{lemma2.6} that $t({T}_{n,1}^{a,b})\geq 1$. Note that $\delta({T}_{n,1}^{a,b})\geq a$. Therefore ${T}_{n,1}^{a,b}\in \mathcal{T}_n(a,b)$ when $b\leq 2a-3$.

For ${T}_{n,2}^{a,b}$ with $a=2$, we can obtain that $|N_{C_2}(X)|=|X|+1>|X|$ for any $\emptyset \neq X\subseteq T$. For ${T}_{n,2}^{a,b}$ with $a\geq 3$, we have $|N_{R_2}(T\setminus \{u_1\})|=b-a+2$, $|N_{R_2}(u_1)|=a-1$, $|N_{R_2}(T\setminus \{u_1\})\cap N_{R_2}(u_1)|=\min\{b-a+2, a-2\}$ and $|N_{R_2}(T)|\geq b-a+3$. By Lemma \ref{lemma2.8}, $$\max_{\emptyset\neq X\subseteq T}(|X|-|N_{R_2}(X)|)=\max\{b-|N_{R_2}(T\setminus \{u_1\})|,b+1-|N_{R_2}(T)|\}=a-2.$$ Then for any $\emptyset \neq X\subseteq T$, we have $|N_{C_2}(X)|=a-1+|N_{R_2}(X)|\geq |X|+1.$ By Lemma \ref{lemma2.6}, $t({T}_{n,2}^{a,b})\geq 1$. Clearly, $\delta({T}_{n,2}^{a,b})\geq a$. Hence ${T}_{n,2}^{a,b}\in \mathcal{T}_n(a,b)$.
\end{proof}

\section{Proof of Theorems \ref{theorem1.3} and \ref{theorem1.4}}
\medskip
\noindent  \textbf{Proof of Theorem \ref{theorem1.3}.}
Suppose that $b\geq 2$ and $b(G)\geq \frac{1}{b}$. Let $X\subseteq V(G)$, and let $I$ be the set of isolated vertices of $G-X$. If $I=\emptyset$, then $i(G-X)=0\leq b|X|$. Otherwise, $N_G(I)\subseteq X$ and $I\cap X=\emptyset$, so $N_G(I)\neq V(G)$. By the definition of the binding number, $$\frac{1}{b}\leq b(G)\leq \frac{|N_G(I)|}{|I|}\leq\frac{|X|}{|I|}.$$ Hence $i(G-X)=|I|\le b|X|$. By Theorem \ref{theorem1.2}, $G$ contains a $\{K_{1,1},K_{1,2},\ldots,K_{1,b}\}$-factor, which is a special $[1,b]$-factor.
\hspace*{\fill}$\Box$

Suppose that $a$, $b$, and $n$ are three positive integers such that $2\leq a<b$ and $n\geq (b+1)^2(2a-1)^2+2b+3$. Suppose that $G$ has the maximum spectral radius among graphs in $\mathcal{B}_n(a,b)$. For convenience, we use $\delta$ instead of $\delta(G)$. In the following, we write $\rho(G)$ as $\rho$ for short. By Lemma \ref{lemma2.2}, there exists $S\subseteq V(G)$ such that
\begin{eqnarray}\label{equation3.1}
\sum_{x\in T}d_{G-S}(x)\leq at-bs-1,
\end{eqnarray}
where $T=\{x\in V(G)\setminus S\mid d_{G-S}(x)<a \}$, $|S|=s$ and $|T|=t$.

Before proving $G\cong \mathcal{B}_n^{a,b}$, which is the main goal, we first prove the following several lemmas, which are very crucial to characterizing the specific structure of $G$.

\begin{lem}\label{lemma3.1}
$s\geq 1$ and $t\geq b+1$. Moreover, $G[V(G)\setminus T]$ is complete and $d_{G}(x)=n-1$ for every $x\in S$.
\end{lem}

\begin{proof}
For every $x\in T$, we have $$a\leq d_{G}(x)=d_{S}(x)+d_{G-S}(x)\leq s+d_{G-S}(x),$$ which implies that $\sum_{x\in T}d_{G-S}(x)\geq at-st.$ By \eqref{equation3.1}, $s(t-b)\geq 1.$ Since $s$ and $t$ are integers, we can deduce that $s\geq 1$ and $t\geq b+1$.

Construct a new graph $G'$ by adding edges such that $G'[V(G')\setminus T]$ is complete and $d_{G'}(x)=n-1$ for every $x\in S$. By Lemma \ref{lemma2.1}, $\rho\leq \rho(G')$. Obviously, $\delta(G')\geq \delta(G)\geq a$ and $b(G')\geq b(G)\geq 1$. Note that $$\sum_{x\in T}d_{G'-S}(x)=\sum_{x\in T}d_{G-S}(x)\leq at-bs-1.$$ By Corollary \ref{corollary2.1}, $G'$ has no $[a,b]$-factor. Then $G'\in \mathcal{B}_n(a,b)$. By the maximality of $\rho$ and Lemma \ref{lemma2.1}, we have $G\cong G'$.
\end{proof}

By Lemma \ref{lemma3.1}, $G$ is connected. Let $\boldsymbol{x} = (x_u)_{u\in V(G)}$ be the Perron vector of $A(G)$.

\begin{lem}\label{lemma3.2}
$t=b+1.$
\end{lem}

\begin{proof}
By Lemma \ref{lemma2.11}, $B_{n,2}^{a,b}\in \mathcal{B}_n(a,b)$. Observe that $\rho(B_{n,2}^{a,b})>\rho(K_{n-b-1})=n-b-2$. Then we can deduce that
\begin{eqnarray}\label{equation3.2}
\rho>n-b-2.
\end{eqnarray}
Suppose to the contrary that $t\geq b+2$. Let $R=V(G)\setminus(S\cup T)$. Since $\sum_{x\in T}d_{G-S}(x)=2|E(G[T])|+|[R,T]|_G$,
\begin{eqnarray*}
e(G)&=&\binom{n-t}{2}+st+|[R,T]|_G+e(G[T])\leq\binom{n-t}{2}+st+\sum_{x\in T}d_{G-S}(x).
\end{eqnarray*}
By \eqref{equation3.1}, we have
\begin{eqnarray}\label{equation3.3}
e(G)\leq\binom{n-t}{2}+at+s(t-b)-1.
\end{eqnarray}
It follows from $\sum_{x\in T}d_{G-S}(x)\geq 0$ that $at-bs\geq 1$, implying $s<\frac{a}{b}t<t$. Combining this with $s\leq n-t$ and \eqref{equation3.3}, we obtain that
\begin{eqnarray}\label{equation3.4}
2e(G)<(n-t)(n-t-1)+2at+2(t-b)\min\{t,n-t\}.
\end{eqnarray}
If $b+2\leq t\leq \frac{n}{2}$, then $\min\{t,n-t\}=t$. It follows from \eqref{equation3.4} that
$2e(G)<(n-t)(n-t-1)+2at+2t(t-b).$ By $n\geq (b+1)^2(2a-1)^2+2b+3$ and $b+2\leq t\leq \frac{n}{2}$, we have
\begin{eqnarray*}
&&(n-b-2)^2-[(n-t)(n-t-1)+2at+2t(t-b)]\\
&\geq&\min\left\{n-(b+2)(2a+5),(\frac{n}{4}-a-b-\frac{7}{2})n+(b+2)^2\right\}>0,
\end{eqnarray*}
which implies that $2e(G)<(n-b-2)^2$. If $t\geq \frac{n}{2}$, then $\min\{t,n-t\}=n-t$. It follows from \eqref{equation3.4} that $2e(G)<(n-t)(n-t-1)+2at+2(n-t)(t-b).$ By $n\geq (b+1)^2(2a-1)^2+2b+3$ and $t\geq \frac{n}{2}$, we have
\begin{eqnarray*}
&&(n-b-2)^2-[(n-t)(n-t-1)+2at+2(n-t)(t-b)]\\
&=&t^2-(2a+2b+1)t+b^2+4b-3n+4\\
&\geq& \frac{1}{4}n^2-(a+b+\frac{7}{2})n+b^2+4b+4>0,
\end{eqnarray*}
implying $2e(G)<(n-b-2)^2$. Since $\rho^2\leq tr(A(G)^2)=2e(G)$, we can obtain that $\rho<n-b-2$, contradicting \eqref{equation3.2}. Hence $t=b+1$.
\end{proof}

\begin{lem}\label{lemma3.3}
$1\leq s\leq a$ and $\sum_{x\in T}d_{G-S}(x)=a(b+1)-bs-1$.
\end{lem}

\begin{proof}
According to Lemma \ref{lemma3.1}, $s\geq 1$. By \eqref{equation3.1}, $\sum_{x\in T}d_{G-S}(x)\geq 0$ and Lemma \ref{lemma3.2}, we have $s\leq \frac{a}{b}t-\frac{1}{b}=a+\frac{a-1}{b}<a+1$. Therefore $1\leq s \leq a$.

It follows from \eqref{equation3.1} and Lemma \ref{lemma3.2} that $\sum_{x\in T}d_{G-S}(x)\leq a(b+1)-bs-1$. Suppose to the contrary that
\begin{eqnarray}\label{equation3.5}
\sum_{x\in T}d_{G-S}(x)\leq a(b+1)-bs-2.
\end{eqnarray}
By $s\geq 1$, we have $a(b+1)-bs-1\leq (b+1)(a-1)$, which implies that $\sum_{x \in T}d_{G-S}(x)<(b+1)(a-1)$. Since $t=b+1$, we can deduce that there exists some vertex $u\in T$ such that $d_{G-S}(u)\leq a-2$. Note that $|V(G)-T-S|=n-s-b-1\geq n-a-b-1>a-2$. Then there exists a vertex $v$ in $V(G)-T-S$ such that $uv\notin E(G)$. Construct a new graph $G'$ from $G$ by adding the edge $uv$. By \eqref{equation3.5}, we have $$\sum_{x\in T}d_{G'-S}(x)=\sum_{x\in T}d_{G-S}(x)+1\leq a(b+1)-bs-1.$$ By Lemma \ref{lemma2.2}, $G'$ has no $[a,b]$-factor. Obviously $\delta(G')\geq a$ and $b(G')\geq 1$. Then $G'\in \mathcal{B}_n(a,b)$. Note that $G$ is a proper spanning subgraph of the connected graph $G'$. It follows from Lemma \ref{lemma2.1} that $\rho(G')>\rho$, contrary to the maximality of $\rho$. Hence $\sum_{x\in T}d_{G-S}(x)=a(b+1)-bs-1$.
\end{proof}

For convenience, we write $d_{G-T}(x)=|N_G(x)\setminus T|$ for each $x\in T$. We can deduce that for every $x\in T$,
\begin{eqnarray*}
d_{G-T}(x)=s+d_{G-S}(x)-d_T(x)\leq s+d_{G-S}(x),
\end{eqnarray*}
with equality for every $x\in T$ if and only if $E(G[T])=\emptyset$. Define $P(G)=\sum_{x\in T}d_{G-T}(x)^2$.

\begin{lem}\label{lemma3.4}
The following bounds hold.\\
(i) If $1\leq s\leq a-1$, then
\begin{eqnarray}\label{equation3.6}
P(G)=\sum_{x\in T}d_{G-T}(x)^2\leq (a+s-1)^2+ba^2.
\end{eqnarray}
Equality holds if and only if $E(G[T])=\emptyset$ and the sequence ($d_{G-S}(x)$: $x\in T$) consists of one entry $a-1$ and $b$ entries $a-s$.\\
(ii) If $s=a$, then $b\leq 2a-2$ and
\begin{eqnarray}\label{equation3.7}
P(G)=\sum_{x\in T}d_{G-T}(x)^2\leq (3a-b-1)^2+(b-a)(a+1)^2+a^3.
\end{eqnarray}
Equality holds if and only if $E(G[T])=\emptyset$ and the sequence ($d_{G-S}(x)$: $x\in T$) consists of one entry $2a-b-1$, $b-a$ entries $1$ and $a$ entries $0$.
\end{lem}

\begin{proof}
Assume first that $1\leq s\leq a-1$. Let
$T=\{v_1,v_2,\dots,v_{b+1}\}$ and $x_i=d_{G-S}(v_i)$, $1\le i\le b+1$. By $\delta\geq a$ and the definition of $T$, we have $a-s\leq x_i\leq a-1$ for every vertex $v_i\in T$. If $s=1$, then $a-s=a-1$, and hence $x_i=a-1$ for every $1\leq i\leq b+1$. Then we have
\begin{eqnarray*}
\sum_{x\in T}d_{G-T}(x)^2\leq \sum_{i=1}^{b+1}(s+x_i)^2=(b+1)a^2=(a+s-1)^2+ba^2
\end{eqnarray*}
with equality if and only if $E(G[T])=\emptyset$ and $d_{G-S}(x)=a-1$ for every $x\in T$. Next consider the case of $2\leq s\leq a-1$. It follows from Lemma \ref{lemma3.3} that $$\sum_{i=1 }^{b+1}x_i=a(b+1)-bs-1=b(a-s)+(a-1).$$ Since $[x_i-(a-s)](a-1-x_i)\geq0$, we have $$a^2+(2a+s-1)[x_i-(a-s)]-(s+x_i)^2\geq 0.$$ Hence $(s+x_i)^2\leq a^2+(2a+s-1)[x_i-(a-s)]$. Then we can obtain that
\begin{eqnarray*}
\sum_{x\in T}d_{G-T}(x)^2&\leq&\sum_{i=1}^{b+1}(s+x_i)^2\leq (b+1)a^2+(2a+s-1)\left[\sum_{i=1}^{b+1}x_i-(b+1)(a-s)\right]\\
&=&(b+1)a^2+(2a+s-1)[b(a-s)+(a-1)-(b+1)(a-s)]\\
&=&(b+1)a^2+(2a+s-1)(s-1)=ba^2+(a+s-1)^2.
\end{eqnarray*}
Equality holds if and only if $[x_i-(a-s)](a-1-x_i)=0$ for $1\leq i\leq b+1$ and $E(G[T])=\emptyset$. Suppose that equality holds. Assume that exactly $k$ of the entries $x_i$ are equal to $a-1$ and the remaining $b+1-k$ entries are equal to $a-s$. Hence $\sum_{i=1}^{b+1}x_i=(b+1)(a-s)+k(s-1)$. Comparing this with $$\sum_{i=1}^{b+1}x_i=b(a-s)+(a-1)=(b+1)(a-s)+(s-1),$$ we obtain $k=1$. Therefore equality holds if and only if $E(G[T])=\emptyset$ and the sequence ($d_{G-S}(x)$: $x\in T$) consists of one entry $a-1$ and $b$ entries $a-s$.

Now assume $s=a$. It follows from Lemma \ref{lemma3.3} that
\begin{eqnarray}\label{equation3.8}
\sum_{x\in T}d_{G-S}(x)=a-1.
\end{eqnarray}
Define $Z_0=\{x\in T: d_{G-S}(x)=0\}$ and $z_0=|Z_0|$. Since $b(G)\geq 1$, we have $z_0\leq |N_{G}(Z_0)|\leq |S|=a,$ which implies that there exist at least $b+1-a$ vertices $x\in T$ such that $d_{G-S}(x)>0$. Together with $\sum_{x\in T}d_{G-S}(x)=a-1$, this implies that $b+1-a\leq b+1-z_0\leq a-1.$ Hence $b\leq 2a-2$ and
\begin{eqnarray}\label{equation3.9}
b-a+2\leq z_0\leq a.
\end{eqnarray}
By the definition of $Z_0$ and $\sum_{x\in T}d_{G-S}(x)=a-1$, we have $1\leq d_{G-S}(x)\leq a-1-(b-z_0),$
where $x\in T\setminus Z_0$. For every $x\in T\setminus Z_0$, by $(d_{G-S}(x)-1)(a-1-b+z_0-d_{G-S}(x))\geq 0$, we can obtain that
\begin{eqnarray}\label{equation3.10}
(a+d_{G-S}(x))^2\leq (3a-b+z_0)d_{G-S}(x)+a^2-a+b-z_0+1.
\end{eqnarray}
By \eqref{equation3.8}, \eqref{equation3.9} and \eqref{equation3.10},
\begin{eqnarray*}
\sum_{x\in T}d_{G-T}(x)^2&\leq&\sum_{x\in T}(s+d_{G-S}(x))^2=\sum_{x\in T}(a+d_{G-S}(x))^2\\
&\leq&(3a-b+z_0)\sum_{x\in T}d_{G-S}(x)+(b+1-z_0)(a^2-a+b-z_0+1)+z_0a^2\\
&=&z_0^2-(2b-2a+3)z_0+a^2b+4a^2-2ab+b^2-4a+3b+1\\
&\leq& b^2+a^2b-4ab+3b+7a^2-7a+1\\
&=&(3a-b-1)^2+(b-a)(a+1)^2+a^3.
\end{eqnarray*}
Equality holds if and only if $z_0=a$, $(d_{G-S}(x)-1)(a-1-b+z_0-d_{G-S}(x))=0$ for every $x\in T\setminus Z_0$ and $E(G[T])=\emptyset$. Suppose that the equality holds.  Let $k$ be the number of vertices of $T\setminus Z_0$ that satisfy $d_{G-S}(x)=1$. The remaining $b+1-a-k$ vertices of $T\setminus Z_0$ satisfy $d_{G-S}(x)=2a-1-b$. Then $$\sum_{x\in T}d_{G-S}(x)=(b-2a+2)k-b^2-2a^2+3ab-2b+3a-1.$$ If $b=2a-2$, it follows from \eqref{equation3.9} that $z_0=a$, which implies that $|T\setminus Z_0|=a-1$. By \eqref{equation3.8}, $k=a-1=b-a+1$. If $b\leq 2a-3$, comparing with $\sum_{x\in T}d_{G-S}(x)=a-1$, we can obtain that $k=b-a$. Hence equality holds if and only if $E(G[T])=\emptyset$ and the sequence ($d_{G-S}(x)$: $x\in T$) consists of one entry $2a-b-1$, $b-a$ entries $1$ and $a$ entries $0$.
\end{proof}

\begin{corollary}\label{corollary3.1}
Every extremal graph $G$ satisfies $s\in\{a-1, a\}$.\\
(i) If $s=a-1$, then $E(G[T])=\emptyset$ and the sequence ($d_{G-S}(x)$: $x\in T$) consists of one entry $a-1$ and $b$ entries $1$.\\
(ii) If $s=a$, then $b\leq 2a-2$, $E(G[T])=\emptyset$ and the sequence ($d_{G-S}(x)$: $x\in T$) consists of one entry $2a-b-1$, $b-a$ entries $1$, and $a$ entries $0$.
\end{corollary}

\begin{proof}
By Lemma \ref{lemma3.3}, $1\leq s\leq a$. For $1\leq s\leq a-1$, the right side of \eqref{equation3.6} is strictly increasing in $s$. In the graph $B_{n,2}^{a,b}$, we take $S=\{v_1,v_2,\dots,v_{a-1}\}$ and $T=\{u_1,u_2,\dots,u_{b+1}\}$. Then we obtain $$P(B_{n,2}^{a,b})=(2a-2)^2+ba^2.$$ Now we claim that $s\geq a-1$. Suppose to the contrary that $s\leq a-2$. By Lemma \ref{lemma3.4}, $$P(G)\leq (a+s-1)^2+ba^2<(2a-2)^2+ba^2=P(B_{n,2}^{a,b}).$$ By Lemma \ref{lemma3.1}, $G[V(G)\setminus T]=K_{n-b-1}$, where $n-b-1\geq (b+1)^2(2a-1)^2+b+2$. Combining the definition of $T$ and $s\leq a-2$, we have $d_{G-T}(x)\leq d_{G-S}(x)+s< 2a-1$. By Lemmas \ref{lemma3.2} and \ref{lemma2.7}, we obtain $\rho< \rho(B_{n,2}^{a,b})$, contradicting the maximality of $\rho$. Hence $s\in \{a-1,a\}$.

If $s=a-1$, we assume to the contrary that the inequality \eqref{equation3.6} is strict. Then $$P(G)<(2a-2)^2+ba^2=P(B_{n,2}^{a,b}).$$ By Lemma \ref{lemma2.7}, $\rho<\rho(B_{n,2}^{a,b})$, contrary to the maximality of $\rho$. Then equality holds in Lemma \ref{lemma3.4}, and hence $E(G[T])=\emptyset$ and the sequence ($d_{G-S}(x)$: $x\in T$) consists of one entry $a-1$ and $b$ entries $1$.

If $s=a$, it follows from Lemma \ref{lemma3.4} that $b\leq 2a-2$. Then $B_{n,1}^{a,b}$ is defined. In the graph $B_{n,1}^{a,b}$, we take $S=\{v_1,v_2,\dots,v_a\}$ and $T=\{u_1,u_2,\dots,u_{b+1}\}$. Note that $$P(B_{n,1}^{a,b})=(3a-b-1)^2+(b-a)(a+1)^2+a^3.$$ Now we prove that the equality in Lemma \ref{lemma3.4} holds. Suppose to the contrary that the inequality in \eqref{equation3.7} is strict. Then we have $$P(G)<(3a-b-1)^2+(b-a)(a+1)^2+a^3=P(B_{n,1}^{a,b}).$$ By Lemma \ref{lemma2.7}, $\rho<\rho(B_{n,1}^{a,b})$, which contradicts the maximality of $\rho$. Hence $E(G[T])=\emptyset$ and the sequence ($d_{G-S}(x)$: $x\in T$) consists of one entry $2a-b-1$, $b-a$ entries $1$ and $a$ entries $0$.
\end{proof}

\begin{lem}\label{lemma3.5}
If $s=a$, then $G\cong B_{n,1}^{a,b}$. If $s=a-1$, then $G\cong B_{n,2}^{a,b}$.
\end{lem}

\begin{proof}
By Lemma \ref{lemma3.2}, $t=b+1$. Let $T=\{u_1,u_2,\dots,u_{b+1}\}$, $S=\{v_1,v_2,\dots,v_s\}$ and $V(G)-S-T=\{w_1,w_2,\dots,w_{n-s-b-1}\}$. By Lemma \ref{lemma3.1}, $d_{G}(x)=n-1$ for every $x\in S$ and $s\geq 1$, and hence $G$ is connected. Let $\boldsymbol{x}=(x_u)_{u\in V(G)}$ be the positive Perron vector of $A(G)$ such that $\sum_{v\in V(G)}x_v^2=1$. Without loss of generality, assume that $x_{u_1}\geq x_{u_2}\geq \cdots\geq x_{u_{b+1}}$, $x_{v_1}\geq x_{v_2}\geq \cdots\geq x_{v_s}$ and $x_{w_1}\geq x_{w_2}\geq \cdots\geq x_{w_{n-s-b-1}}$.

First assume $s=a$. By Corollary \ref{corollary3.1}, we have $b\leq 2a-2$, $E(G[T])=\emptyset$ and the sequence ($d_{G-S}(x)$: $x\in T$) consists of one entry $2a-b-1$, $b-a$ entries $1$ and $a$ entries $0$. Since $b\leq 2a-2$, we have $2a-b-1\geq 1$. By $x_{u_1}\geq x_{u_2}\geq \cdots\geq x_{u_{b+1}}$, we have $d_{G-S}(u_{b-a+2})=\cdots=d_{G-S}(u_{b+1})=0$. By symmetry, we have $x_{u_{b-a+2}}=\cdots=x_{u_{b+1}}$. Let $Z=\{u_{b-a+2},\dots,u_{b+1}\}.$ By $b(G)\geq 1$, we can deduce that for any subset $X\subseteq \{u_1,u_2,\dots,u_{b-a+1}\}$,
\begin{eqnarray}\label{equation3.11}
|N_{G-S}(X)|=|N_{G-S}(X\cup Z)|\geq |X\cup Z|-a=|X|.
\end{eqnarray}
If $2a-b-1=1$, it follows from \eqref{equation3.11} that $|N_{G-S}(u_i)|=1$ $(1\leq i
\leq b-a+1)$ and $N_{G-S}(u_i)\cap N_{G-S}(u_j)=\emptyset$ $(1\leq i<j\leq b-a+1)$. By the maximality of $\rho$,
\[N_{G-S}(u_{i})=
\begin{cases}
\{w_i\},&1\leq i\leq b-a+1,\\
\emptyset,&b-a+2\leq i\leq b+1,
\end{cases}
\]
and hence $G\cong B_{n,1}^{a,b}$. Now suppose that $2a-b-1\geq 2$ and $d_{G-S}(u)=2a-b-1$ $(u\in T\setminus Z)$. According to Lemma \ref{lemma2.9}, $x_u>x_{u'}$ for any $u'\in T\setminus (Z\cup\{u\})$. By $A(G)\boldsymbol{x}=\rho\boldsymbol{x}$, for any $u'\in T\setminus Z$ and $u''\in Z$, $$\rho(x_{u'}-x_{u''})=\sum_{w\in N_{G-S}(u')}x_w>0,$$ which implies that  $d_{G-S}(u_1)=2a-b-1$ and $d_{G-S}(u_i)=1$ where $2\leq i\leq b-a+1$. For any $X\subseteq T\setminus Z$, it follows from $b(G)\geq 1$ and Lemma \ref{lemma2.5} that $a+|N_{G-S}(X)|=|N_G(X\cup Z)|\geq |X|+a,$ which implies that
\begin{eqnarray}\label{equation3.12}
|N_{G-S}(X)|\geq |X|.
\end{eqnarray}
Hence $N_{G-S}(u_i)\cap N_{G-S}(u_j)=\emptyset$ $(2\leq i<j\leq b-a+1)$. Let $Q=\cup_{i=2}^{b-a+1}N_{G-S}(u_i)$. By \eqref{equation3.12}, $N_{G-S}(u_1)\setminus Q\neq \emptyset$, implying $|N_{G-S}(u_1)\cap Q|\leq \min\{2a-b-2,b-a\}$. Next we prove that $|N_{G-S}(u_1)\cap Q|= \min\{2a-b-2,b-a\}$. Suppose to the contrary that $|N_{G-S}(u_1)\cap Q|< \min\{2a-b-2,b-a\}$. Since $|N_{G-S}(u_1)\cap Q|<2a-b-2$ and $|N_{G-S}(u_1)\cap Q|<b-a$, there exist $w_j\in Q\setminus N_{G-S}(u_1)$ and $w_l\in N_{G-S}(u_1)\setminus Q$. Suppose $N_{G-S}(u_i)=\{w_j\}$ where $2\leq i\leq b-a+1$. By $A(G)\boldsymbol{x}=\rho\boldsymbol{x}$, $$(\rho+1)(x_{w_l}-x_{w_j})=x_{u_1}-x_{u_i}>0,$$ implying $x_{w_l}>x_{w_j}$. Construct $G'=G-u_iw_j+u_iw_l$. For every $\emptyset\neq X\subseteq T$, $$|N_{G'}(X)|=a+|N_{G'-S}(X\setminus Z)|\geq a+|X\setminus Z|\geq |X|.$$
It follows from Lemma \ref{lemma2.5} that $b(G')\geq 1$. It is easy to check that $\delta(G')\geq a$ and $G'$ has no $[a,b]$-factor. So $G'\in \mathcal{B}_n(a,b).$ By $x_{w_l}>x_{w_j}$ and Lemma \ref{lemma2.3}, $\rho<\rho(G')$, contradicting the maximality of $\rho$. Then $|N_{G-S}(u_1)\cap Q|= \min\{2a-b-2,b-a\}$. By the maximality of $\rho$, we can deduce that
\[N_{G-S}(u_1)=\{w_1,w_2,\dots,w_{2a-b-1}\},\]
\[N_{G-S}(u_i)=
\begin{cases}
\{w_{i-1}\},&2\leq i\leq |N_{G-S}(u_1)\cap Q|+1,\\
\{w_{2a-b+i-|N_{G-S}(u_1)\cap Q|-2}\},&|N_{G-S}(u_1)\cap Q|+2\leq i\leq b-a+1,\\
\emptyset,&b-a+2\leq i\leq b+1.
\end{cases}
\]
Hence $G\cong B_{n,1}^{a,b}$.

Next assume $s=a-1$. By Corollary \ref{corollary3.1}, $E(G[T])=\emptyset$ and the sequence ($d_{G-S}(x)$: $x\in T$) consists of one entry $a-1$ and $b$ entries $1$. First suppose that $a=2$. Observe that every vertex $u\in T$ satisfies $|N_{G-S}(u)|=1$. Since $b(G)\geq 1$, we have $\max_{X\subseteq T}(|X|-|N_{G-S}(X)|)=b+1-|N_{G-S}(T)|\leq 1,$ which implies
\begin{eqnarray*}
b\leq |N_{G-S}(T)|\leq b+1.
\end{eqnarray*}
Suppose to the contrary that $|N_{G-S}(T)|= b+1$. Then there exist two vertices $u_i,u_j\in T$ such that $N_{G-S}(u_i)=\{w_k\}$ and $N_{G-S}(u_j)=\{w_l\}$ where $i\neq j$ and $k<l$. Let $G'=G-u_jw_l+u_jw_k$. Observe that $\delta(G')\geq a$ and $G'$ satisfies \eqref{equation1.1}. By Corollary \ref{corollary2.1}, $G'$ has no $[a,b]$-factor. For any $\emptyset \neq X\subseteq T$, $|N_{G'}(X)|=|N_{G'-S}(X)|+1\geq |X|$. According to Lemma \ref{lemma2.5}, $b(G')\geq 1$, and hence $G'\in \mathcal{B}_n(a,b)$. Combining $x_{w_k}\geq x_{w_l}$ $(k<l)$ and Lemma \ref{lemma2.3}, $\rho<\rho(G')$, contrary to the maximality of $\rho$. Therefore $|N_{G-S}(T)|= b$. Suppose that $N_{G-S}(u_i)=N_{G-S}(u_j)=\{w_p\}$ where $u_i,u_j\in T$ and $i\neq j$. For any vertex $u_l\in T\setminus \{u_i,u_j\}$, there exists a vertex $w_q$ $(p\neq q)$ such that $N_{G-S}(u_l)=\{w_q\}$. By symmetry, $x_{u_i}=x_{u_j}$. By $A(G)\boldsymbol{x}=\rho \boldsymbol{x}$, we have $$[\rho(\rho+1)-1](x_{u_i}-x_{u_l})=x_{u_i}>0,$$ and hence $$(\rho+1)(x_{w_p}-x_{w_q})=x_{u_i}+(x_{u_i}-x_{u_l})>0.$$ Then we obtain that $x_{u_i}=x_{u_j}>x_{u_l}$ and $x_{w_p}>x_{w_q}$. Since $u_l$ is arbitrary, we have $N_{G-S}(u_1)=N_{G-S}(u_2)=\{w_1\}$. By the maximality of $\rho$, $N_{G-S}(u_i)=\{w_{i-1}\}$ $(3\leq i\leq b+1)$, and hence $G\cong B_{n,2}^{2,b}$. Assume $a\geq 3$. Then $d_{G-S}(u_1)=a-1\geq2$ and $d_{G-S}(u_i)=1$ where $2\leq i\leq b+1$. By Lemma \ref{lemma2.9}, for every $2\leq i\leq b+1$,
\begin{eqnarray*}
x_{u_1}>x_{u_i}.
\end{eqnarray*}
Let $Q=\cup_{i=2}^{b+1}N_{G-S}(u_i).$ Since $b(G)\geq 1$, we have
\begin{eqnarray*}
|Q|\geq b-a+1,\quad |N_{G-S}(T)|\geq b-a+2,
\end{eqnarray*}
and then $G\in \mathcal{G}_n^{a,b}(b-a+1)$. By Lemma \ref{lemma2.10}, $B_{n,2}^{a,b}$ is the unique graph with maximum spectral radius in $\mathcal{G}_n^{a,b}(b-a+1)$. It follows from Lemma \ref{lemma2.11} that $G\cong B_{n,2}^{a,b}$.
\end{proof}

Now we are ready to provide the proof of Theorem \ref{theorem1.4}.

\medskip
\noindent  \textbf{Proof of Theorem \ref{theorem1.4}.}
By Lemma \ref{lemma2.11}, $B_{n,2}^{a,b}$ belongs to $\mathcal{B}_n(a,b)$ and $B_{n,1}^{a,b}$ also belongs to $\mathcal{B}_n(a,b)$ when $b\leq 2a-2$. By Corollary \ref{corollary3.1} and Lemma \ref{lemma3.5}, we have $G\cong B_{n,2}^{a,b}$ when $s=a-1$ and $G\cong B_{n,1}^{a,b}$ when $s=a$, $b\leq 2a-2$. If $b>2a-2$, then we prove $s=a-1$. Suppose to the contrary that $s=a$. By Lemma \ref{lemma3.3}, $\sum_{x\in T}d_{G-S}(x)=a-1$, which implies
\begin{eqnarray}\label{equation3.13}
|N_{G-S}(T)|\leq \sum_{x\in T}d_{G-S}(x)=a-1.
\end{eqnarray}
Since $b(G)\geq 1$, we have
\begin{eqnarray}\label{equation3.14}
|N_{G-S}(T)|\geq |T|-s=b-a+1.
\end{eqnarray}
Combining \eqref{equation3.13} and \eqref{equation3.14}, we obtain $b\leq 2a-2$, which contradicts $b>2a-2$. Then $s=a-1$, and hence $G\cong B_{n,2}^{a,b}$. Now assume that $b\leq 2a-2$. Recall the definitions of $B_{n,1}^{a,b}$ and $B_{n,2}^{a,b}$. Combining $P(B_{n,1}^{a,b})=(3a-b-1)^2+(b-a)(a+1)^2+a^3$ and $P(B_{n,2}^{a,b})=(2a-2)^2+ba^2$, we obtain that
\begin{eqnarray}\label{equation3.15}
P(B_{n,1}^{a,b})-P(B_{n,2}^{a,b})=b^2+3a^2-4ab+3b+a-3\equiv 1 (\rm{mod}~2),
\end{eqnarray}
and hence $b^2+3a^2-4ab+3b+a-3\neq 0$. Assume that $b^2+3a^2-4ab+3b+a-3>0$. By \eqref{equation3.15} and Lemma \ref{lemma2.7}, $\rho(B_{n,1}^{a,b})>\rho(B_{n,2}^{a,b})$. If $b^2+3a^2-4ab+3b+a-3<0$, then it follows from \eqref{equation3.15} that $P(B_{n,1}^{a,b})<P(B_{n,2}^{a,b})$. By Lemma \ref{lemma2.7}, $\rho(B_{n,1}^{a,b})<\rho(B_{n,2}^{a,b})$. Therefore $\rho(B_{n,1}^{a,b})>\rho(B_{n,2}^{a,b})$ holds if and only if $3a^2-4ab+a+b^2+3b-3>0$. By the maximality of $\rho$, it follows that $G\cong B_{n,1}^{a,b}$ when $3a^2-4ab+a+b^2+3b-3>0$ and $G\cong B_{n,2}^{a,b}$ otherwise. Hence $G\cong \mathcal{B}_n^{a,b}$.
\hspace*{\fill}$\Box$

\section{Proof of Theorem \ref{theorem1.7}}
Assume that $a$, $b$, and $n$ are three positive integers such that $2\leq a<b$ and $n\geq (b+1)^2(2a-1)^2+2b+3$. Suppose that $G^*$ has the maximum spectral radius among graphs in $\mathcal{T}_n(a,b)$. For convenience, we use $\delta^*$ instead of $\delta(G^*)$ and write $\rho(G^*)$ as $\rho^*$ for short. By Lemma \ref{lemma2.2}, there exists $S\subseteq V(G^*)$ such that
\begin{eqnarray}\label{equation4.1}
\sum_{x\in T}d_{G^*-S}(x)\leq at-bs-1,
\end{eqnarray}
where $T=\{x\in V(G^*)\setminus S\mid d_{G^*-S}(x)<a \}$, $|S|=s$ and $|T|=t$.

Combining $\delta^*\geq a$ and \eqref{equation4.1}, we have $s(t-b)\geq 1$, which implies that $s\geq 1$ and $t\geq b+1$. By Lemma \ref{lemma2.4}, we know that adding edges preserves the 1-tough property of
a graph. Furthermore, construct a new graph $G'$ from $G^*$ by adding edges such that $G'[V(G')\setminus T]=K_{n-t}$ and $d_{G'}(x)=n-1$ for every $x\in S$. Clearly, $G'$ satisfies \eqref{equation4.1} and $\delta(G')\geq a$. By the same arguments, Lemma \ref{lemma3.1} still holds, and hence $G^*[V(G^*)\setminus T]$ is complete and $d_{G^*}(x)=n-1$ for every $x\in S$. By Lemmas \ref{lemma2.1} and \ref{lemma2.12}, $\rho^*\geq \rho({T}_{n,2}^{a,b})>\rho(K_{n-b-1})=n-b-2$. Combining this with Lemma \ref{lemma3.1} and \eqref{equation4.1}, Lemma \ref{lemma3.2} still holds. Hence we obtain
\begin{eqnarray}\label{equation4.2}
t=b+1.
\end{eqnarray}
Let $T=\{u_1,\dots,u_{b+1}\}$. Since $d_{G^*}(x)=n-1$ for every $x\in S$ and $s\geq 1$, $G^*$ is connected. Let $\boldsymbol{x}=(x_v)_{v\in V(G^*)}$ be a positive Perron vector of $A(G^*)$ such that $\sum_{v\in V(G^*)}x_v^2=1$. Without loss of generality, assume that $x_{u_1}\geq \cdots\geq x_{u_{b+1}}$. According to $0\leq \sum_{x\in T}d_{G^*-S}(x)\leq at-bs-1$, we have $s\leq \frac{a}{b}t-\frac{1}{b}=a+\frac{a-1}{b}<a+1,$ implying
\begin{eqnarray}\label{equation4.3}
1\leq s\leq a.
\end{eqnarray}
It follows from \eqref{equation4.1} and \eqref{equation4.2} that $\sum_{x\in T}d_{G^*-S}(x)\leq a(b+1)-bs-1$. Now we claim that $\sum_{x\in T}d_{G^*-S}(x)=a(b+1)-bs-1.$ Suppose to the contrary that $\sum_{x\in T}d_{G^*-S}(x)\leq a(b+1)-bs-2$. By \eqref{equation4.2}, \eqref{equation4.3} and $n\geq (b+1)^2(2a-1)^2+2b+3$, there exist $u\in T$, $v\in V(G^*)\setminus (S\cup T)$ such that $d_{G^*-S}(u)\leq a-2$ and $uv\notin E(G^*)$. Let $G'=G^*+uv$. Clearly, $G'$ satisfies \eqref{equation1.1} and hence $G'$ has no $[a,b]$-factor. By Lemma \ref{lemma2.4}, $t(G')\geq t(G^*)\geq 1$. Note that $\delta(G')\geq a$. So $G'\in \mathcal{T}_n(a,b)$. By Lemma \ref{lemma2.1}, $\rho(G')>\rho^*$, contrary to the maximality of $\rho^*$. Therefore
\begin{eqnarray}\label{equation4.4}
\sum_{x\in T}d_{G^*-S}(x)=a(b+1)-bs-1.
\end{eqnarray}
Define $P(G^*)=\sum_{x\in T}d_{G^*-T}(x)^2$.

We first prove the following lemma, which plays an important role in characterizing the specific structure of $G^*$.

\begin{lem}\label{lemma4.1}
$s\in \{a-1,a\}$. If $s=a-1$, $G^*\cong {T}_{n,2}^{a,b}$. If $s=a$, then $b\leq 2a-3$ and $G^*\cong {T}_{n,1}^{a,b}$.
\end{lem}

\begin{proof}
By \eqref{equation4.3}, $1\leq s\leq a$. According to $\delta^*\geq a$ and the definition of $T$, we have $a-s\leq d_{G^*-S}(u)\leq a-1$ for every $u\in T$. Assume to the contrary that $s\leq a-2$. By \eqref{equation4.4} and the calculation in the proof of Lemma \ref{lemma3.4}(i), we have
\begin{eqnarray*}
P(G^*)\leq (a+s-1)^2+ba^2< (2a-2)^2+ba^2=P({T}_{n,2}^{a,b}).
\end{eqnarray*}
Combining Lemmas \ref{lemma2.7} and \ref{lemma2.12}, we have $\rho^*<\rho({T}_{n,2}^{a,b})$, which contradicts the maximality of $\rho^*$. Therefore $s\in \{a-1,a\}$.

Suppose $s=a-1$. Let $S=\{v_1,\dots,v_{a-1}\}$ and $V(G^*)\setminus (S\cup T)=\{w_1,\dots,w_{n-a-b}\}$. Without loss of generality, assume that $x_{v_1}\geq \cdots\geq x_{v_{a-1}}$ and $x_{w_1}\geq \cdots\geq x_{w_{n-a-b}}$. Note that
\begin{eqnarray}\label{equation4.5}
P(G^*)\leq(2a-2)^2+ba^2=P({T}_{n,2}^{a,b}).
\end{eqnarray}
By the maximality of $\rho^*$ and Lemma \ref{lemma2.7}, equality holds in \eqref{equation4.5}. By the same arguments, Lemma \ref{lemma3.4}(i) still holds, and hence $E(G^*[T])=\emptyset$ and the sequence $(d_{G^*-S}(x):x\in T)$ consists of one entry $a-1$ and $b$ entries $1$. If $a=2$, then $d_{G^*-S}(u)=1$ for every $u\in T$. Since $G^*[V(G^*)\setminus T]=K_{n-b-1}$ and $n-b-1\geq t+2$, it follows from Lemma \ref{lemma2.6} that for any $\emptyset \neq X\subseteq T$, $|N_{G^*-S}(X)|\geq |X|$. Then $G^* \cong T_{n,2}^{2,b}$. If $a\geq 3$, then there exists a unique vertex $u\in T$ such that $d_{G^*-S}(u)\geq 2$ and every vertex $u'\in T\setminus \{u\}$ satisfies $d_{G^*-S}(u')=1$. By Lemma \ref{lemma2.9}, $x_u>x_{u'}$, which implies that $u=u_1$. Hence $d_{G^*-S}(u_1)=a-1$ and $d_{G^*-S}(u_i)=1$ where $2\leq i\leq b+1$. Since $G^*[V(G^*)\setminus T]=K_{n-b-1}$ and $n-b-1\geq t+2$, it follows from Lemma \ref{lemma2.6} that for any $\emptyset \neq X\subseteq T$, $a-1+|N_{G^*-S}(X)|\geq |X|+1$. Let $L=T\setminus \{u_1\}$. By Lemma \ref{lemma2.8}, we have $|N_{G^*-S}(L)|\geq b-a+2$ and $|N_{G^*-S}(T)|\geq b-a+3$. So $G^*\in \mathcal{G}_n^{a,b}(b-a+2)$. Suppose that $G\in \mathcal{G}_n^{a,b}(b-a+2)$. By the definition of $\mathcal{G}_n^{a,b}(b-a+2)$, Lemmas \ref{lemma2.6} and \ref{lemma2.8}, $t(G)\geq 1$. Observe that $\delta(G)\geq a$ and $G$ satisfies \eqref{equation1.1}. It follows from Corollary \ref{corollary2.1} that $G$ has no $[a,b]$-factor, and then $G\in \mathcal{T}_n(a,b)$. Hence $\mathcal{G}_n^{a,b}(b-a+2)\subseteq \mathcal{T}_n(a,b)$. By the maximality of $\rho^*$ and Lemma \ref{lemma2.10}, $G^*\cong {T}_{n,2}^{a,b}$.

If $s=a$, let $Z=\{x\in T: d_{G^*-S}(x)=0\}$. Let $S=\{v_1,\dots,v_{a}\}$ and $V(G^*)\setminus (S\cup T)=\{w_1,\dots,w_{n-a-b-1}\}$. Without loss of generality, assume that $x_{v_1}\geq \cdots\geq x_{v_{a}}$ and $x_{w_1}\geq \cdots\geq x_{w_{n-a-b-1}}$. It follows from \eqref{equation4.4} that $\sum_{x\in T}d_{G^*-S}(x)=a-1$, and then $|Z|\geq b-a+2>0$. By $t(G^*)\geq 1$ and the definition of $Z$, we have $b-a+2\leq |Z|\leq a-1$ implying $b\leq 2a-3$. For any $u\in T\setminus Z$, we have $$1\leq d_{G^*-S}(u)\leq a-1-(b-|Z|)\leq 2a-b-2,$$ and then $[d_{G^*-S}(u)-1][(2a-b-2)-d_{G^*-S}(u)]\geq 0$ implying $$d_{G^*-S}(u)^2\leq (2a-b-1)d_{G^*-S}(u)-(2a-b-2).$$ Observe that $P({T}_{n,1}^{a,b})=(3a-b-2)^2+(b-a+1)(a+1)^2+(a-1)a^2$. Then
\begin{eqnarray}\label{equation4.6}
P(G^*)&\leq& \sum_{x\in T}(a+d_{G^*-S}(x))^2\nonumber\\
&\leq&(b+1)a^2+(4a-b-1)(a-1)-(2a-b-2)(b+1-|Z|)\nonumber\\
&=&P({T}_{n,1}^{a,b})-(2a-b-2)(a-1-|Z|)\leq P({T}_{n,1}^{a,b}).
\end{eqnarray}
By Lemmas \ref{lemma2.7}, \ref{lemma2.12} and the maximality of $\rho^*$, the equality holds in \eqref{equation4.6}. Hence $E(G^*[T])=\emptyset$, $|Z|=a-1$ and $d_{G^*-S}(u)\in \{1, 2a-b-2\}$ for every $u\in T\setminus Z$. Let $k$ be the number of vertices $u\in T\setminus Z$ such that $d_{G^*-S}(u)=2a-b-2$. Since $|T\setminus Z|=b-a+2$ and $\sum_{x\in T}d_{G^*-S}(x)=a-1$, we have $$k(2a-b-2)+(b-a+2-k)=a-1.$$ If $2a-b-3=0$, then $|T\setminus Z|=b-a+2=a-1$. Then $k=a-1$, and hence $d_{G^*-S}(u_i)=2a-b-2=1$, where $1\leq i\leq b-a+2$. If $2a-b-3\geq 1$, then $k=1$. Suppose that $d_{G^*-S}(u)=2a-b-2\geq 2$ $(u\in T)$. By Lemma \ref{lemma2.9}, $x_u>x_{u'}$ for any $u'\in T\setminus (Z\cup\{u\})$. According to $A(G^*)\boldsymbol{x}=\rho^*\boldsymbol{x}$, for any $u'\in T\setminus Z$ and $u''\in Z$, $$\rho^*(x_{u'}-x_{u''})=\sum_{w\in N_{G^*-S}(u')}x_w>0,$$ and hence $d_{G^*-S}(u_1)=2a-b-2$, $d_{G^*-S}(u_i)=1$ $(2\leq i\leq b-a+2)$ and $d_{G^*-S}(u_j)=0$ $(b-a+3\leq j\leq b+1)$. Combining $G^*[V(G^*)\setminus T]=K_{n-b-1}$, $n-b-1\geq t+2$, $t(G^*)\geq 1$ and Lemma \ref{lemma2.6}, we obtain that for any $\emptyset \neq X\subseteq T$,
\begin{eqnarray}\label{equation4.7}
a+|N_{G^*-S}(X)|\geq |X|+1.
\end{eqnarray}
Then for any $X\subseteq T\setminus Z$,
\begin{eqnarray}\label{equation4.8}
|N_{G^*-S}(X)|=|N_{G^*-S}(X\cup Z)|\geq |X|.
\end{eqnarray}
Therefore $u_2,\dots,u_{b-a+2}$ have distinct single neighbors in $V(G^*)\setminus (S\cup T)$. Let $Q=N_{G^*-S}(\{u_2,\dots,u_{b-a+2}\})$. It follows from \eqref{equation4.8} that
$$|N_{G^*-S}(u_1)\setminus Q|=|N_{G^*-S}(T\setminus Z)|-|Q|\geq 1,$$
and hence $N_{G^*-S}(u_1)\setminus Q\neq \emptyset$. Conversely, if $u_2,\dots,u_{b-a+2}$ have distinct single neighbors in $V(G^*)\setminus (S\cup T)$ and $N_{G^*-S}(u_1)\setminus Q\neq \emptyset$, we can deduce that \eqref{equation4.7} holds, which implies that $t(G^*)\geq 1$. If $2a-b-3=0$, then $u_1,u_2,\dots,u_{b-a+2}$ have distinct single neighbors in $V(G^*)\setminus (S\cup T)$, so $G^*\cong {T}_{n,1}^{a,b}$. Assume $2a-b-3\geq 1$. Observe that $|Q\cap N_{G^*-S}(u_1)|\leq \min\{b-a+1,2a-b-3\}$. Next we prove that $|Q\cap N_{G^*-S}(u_1)|=\min\{b-a+1,2a-b-3\}$. Suppose to the contrary that $|Q\cap N_{G^*-S}(u_1)|<\min\{b-a+1,2a-b-3\}$. Since $|Q\cap N_{G^*-S}(u_1)|<b-a+1$ and $|Q\cap N_{G^*-S}(u_1)|<2a-b-3$, there exist $w_j\in Q\setminus N_{G^*-S}(u_1)$ and $w_l\in N_{G^*-S}(u_1)\setminus Q$. Let $N_{G^*-S}(u_i)=\{w_j\}$ for some $2\leq i\leq b-a+2$. Let $G'=G^*-u_iw_j+u_iw_l$. Note that $u_2,\dots,u_{b-a+2}$ have distinct single neighbors in $V(G')\setminus (S\cup T)$ and $N_{G'-S}(u_1)\setminus (Q\cup \{w_l\})\neq \emptyset$. Hence $t(G')\geq 1$. Clearly, $\delta(G')\geq a$ and $G'$ satisfies \eqref{equation1.1}. By Corollary \ref{corollary2.1}, $G'$ has no $[a,b]$-factor, and hence $G'\in \mathcal{T}_n(a,b)$. By $A(G^*)\boldsymbol{x}=\rho^*\boldsymbol{x}$, $(\rho^*+1)(x_{w_l}-x_{w_j})=x_{u_1}-x_{u_i}>0,$ implying $x_{w_l}>x_{w_j}$. By Lemma \ref{lemma2.3}, $\rho^*<\rho(G')$, which contradicts the maximality of $\rho^*$. Then $|Q\cap N_{G^*-S}(u_1)|=\min\{b-a+1,2a-b-3\}$, which implies that $G^*\cong{T}_{n,1}^{a,b}$.
\end{proof}

Now we are ready to provide the proof of Theorem \ref{theorem1.7}.

\medskip
\noindent  \textbf{Proof of Theorem \ref{theorem1.7}.}
By Lemma \ref{lemma2.12}, ${T}_{n,2}^{a,b}\in \mathcal{T}_n(a,b)$ and ${T}_{n,1}^{a,b}$ also belongs to $\mathcal{T}_n(a,b)$ when $b\leq 2a-3$. By Lemma \ref{lemma4.1}, we have $G^*\cong {T}_{n,2}^{a,b}$ when $s=a-1$ and $G^*\cong {T}_{n,1}^{a,b}$ when $s=a$ and $b\leq 2a-3$. If $b>2a-3$, then we prove $s=a-1$. Suppose to the contrary that $s=a$. It follows from \eqref{equation4.4} that $\sum_{x\in T}d_{G^*-S}(x)=a-1.$ Since $b>2a-3$ and $G^*[V(G^*)\setminus (S\cup T)]$ is a nonempty clique, we obtain
\begin{eqnarray*}
c(G^*-S)&\geq& t+1-e(G^*[T])-|[V(G^*)\setminus (S\cup T),T]|_{G^*}\\
&\geq& t+1-\sum_{x\in T}d_{G^*-S}(x)=b-a+3>a=|S|,
\end{eqnarray*}
contradicting $t(G^*)\geq 1$. Therefore $s=a-1$. So $G^*\cong {T}_{n,2}^{a,b}$ when $b>2a-3$. Now assume that $b\leq 2a-3$. Recall the definitions of ${T}_{n,1}^{a,b}$ and ${T}_{n,2}^{a,b}$. Note that
\begin{eqnarray*}
P({T}_{n,1}^{a,b})-P({T}_{n,2}^{a,b})=3a^2-4ab-3a+b^2+5b+1 \equiv 1 (\rm{mod}~2).
\end{eqnarray*}
Hence $3a^2-4ab-3a+b^2+5b+1\neq 0$. If $3a^2-4ab-3a+b^2+5b+1>0$, then $P({T}_{n,1}^{a,b})>P({T}_{n,2}^{a,b})$. By Lemma \ref{lemma2.7}, $\rho({T}_{n,1}^{a,b})>\rho({T}_{n,2}^{a,b})$. Similarly, we can deduce that $\rho({T}_{n,1}^{a,b})<\rho({T}_{n,2}^{a,b})$ when $3a^2-4ab-3a+b^2+5b+1<0$. Hence $\rho({T}_{n,1}^{a,b})>\rho({T}_{n,2}^{a,b})$ if and only if $3a^2-4ab-3a+b^2+5b+1>0$. By the maximality of $\rho^*$, it follows that $G^*\cong {T}_{n,1}^{a,b}$ when $b\leq 2a-3$ and $3a^2-4ab-3a+b^2+5b+1>0$ and $G^*\cong {T}_{n,2}^{a,b}$ when $b\leq 2a-3$ and $3a^2-4ab-3a+b^2+5b+1<0$. Hence $G^*\cong \mathcal{T}_n^{a,b}$.
\hspace*{\fill}$\Box$

\vspace{14mm}
\noindent
{\bf Declaration of competing interest}
\vspace{3mm}

The authors declare that they have no known competing financial interests or personal relationships that could have appeared to influence the work reported in this paper.

\vspace{4mm}
\noindent
{\bf Data availability}
\vspace{2mm}

No data were used for the research described in this paper.

\end{document}